\documentclass[11pt, twoside, a4paper, reqno]{amsart}
\usepackage{amsmath, amssymb, latexsym, amsthm, amsfonts, mathrsfs, amscd}
\DeclareMathAlphabet{\mathpzc}{OT1}{pzc}{m}{it}
\usepackage{enumerate}[1.)]
\renewcommand{\eqref}[1]{(\ref{#1})}   
\usepackage{latexsym}
\usepackage{euscript}
\usepackage[top =2.1cm, bottom = 2.1cm, left=2.1cm, right=2.1cm]{geometry}
\usepackage{epsfig}
\usepackage{hyperref}
\usepackage{graphics}
\usepackage{array}
\usepackage{enumerate}
\usepackage{url}
\usepackage[utf8]{inputenc}
\usepackage{color}

\theoremstyle{definition} 
\newtheorem{theorem}{Theorem}[section]
\newtheorem{lemma}[theorem]{Lemma}
\newtheorem{proposition}[theorem]{Proposition}
\newtheorem{corollary}[theorem]{Corollary}

\newtheorem{remark}{Remark}[section]
\newcommand{\starsum}{\sideset{}{^{\star}} \sum}
\renewcommand{\l}{\left}
\renewcommand{\r}{\right}
\newcommand{\noi}{\noindent}
\newcommand{\moda}{(\textnormal{mod }a)}

\newcommand{\diff}{\textnormal{d}}

\newcommand{\R}{{\mathbb R}}
\newcommand{\Z}{{\mathbb Z}}
\newcommand{\C}{{\mathbb C}}
\newcommand{\N}{{\mathbb N}}
\newcommand{\h}{{\mathbb H}}
\newcommand{\abcdsumX}{\mathop{\sum_a\sum_b\sum_c\sum_d}_{\substack{ad-bc=r\\|a|, |b|, |c|, |d|\leq X}}}
\newcommand{\abcdsumrX}{\mathop{\sum_a\sum_b\sum_c\sum_d}_{\substack{ad-bc=r\\1\leq a, b, c, d \leq X}}}
\newcommand{\abcdsumoneX}{\mathop{\sum_a\sum_b\sum_c\sum_d}_{\substack{ad-bc=1\\1\leq a, b, c, d \leq X}}}
\newcommand{\abcdsumV}{\mathop{\sum_a\sum_b\sum_c\sum_d}_{\substack{ad-bc=r }}}

\newcommand{\be}{\begin{equation}}
\newcommand{\ee}{\end{equation}}
\newcommand{\ba}{\begin{equation}\begin{aligned}}
\newcommand{\ea}{\end{aligned}\end{equation}}

\newcommand{\ve}{\varepsilon}

\newcommand{\ts}{\mathop{\sum_{a_1 > 0} \sum_{0 \leq j \leq k} \sum_{1 \leq n \ll \frac{X^{1+\ve}}{lH} }\sum_{ 1 \leq m \ll \frac{X^{1+\ve}}{H} }}}

\numberwithin{equation}{section}
\makeatletter
\@namedef{subjclassname@2020}{\textup{2020} Mathematics Subject Classification}
\makeatother

\begin{document}
\title[Lattice points on determinant surfaces]  {Lattice points on determinant surfaces and the spectrum of the automorphic Laplacian} 
\author{Satadal Ganguly, Rachita Guria}
	\address{Theoretical Statistics and Mathematics Unit, Indian Statistical Institute, 203 Barrackpore
		Trunk Road, Kolkata-700108, India.}
	\email{sgisical@gmail.com}
\address{Theoretical Statistics and Mathematics Unit, Indian Statistical Institute, 203 Barrackpore
	Trunk Road, Kolkata-700108, India.}
	\email{guriarachita2011@gmail.com}
	
\maketitle	

\begin{abstract}
	We use classical Fourier analysis along with tools from the spectral theory of Automorphic forms to derive an asymptotic formula with a strong error term for the number of integer solutions $(a, b, c, d)$ inside the expanding box $[-X,X]^4$ to the determinant equation $ad-bc=r$, where $r \neq 0$ is a fixed integer. Furthermore, we apply our method to study sums over these solutions where the variables are weighted; by periodic arithmetical functions in two of the variables in one case, and by an arbitrary sequence of complex numbers in another.
\end{abstract}

\section{Introduction}

The famous Gauss circle problem asks for a precise count for the number $N(r)$ of lattice points, i.e., points with integer coordinates, in the Euclidean plane inside a circle of 
radius $r$ as $r$ increases to infinity. Obtaining the asymptotic relation $N(r)\sim \pi r^2$ (as $r\longrightarrow \infty$) is an elementary exercise and the real task  is to estimate the growth of the error term $E(r)= N(r)-\pi r^2$. The current record is
\[
E(r)=O(r^{131/208+\ve}), \ \text{for any } \ve>0,
\]
due to Huxley \cite{H} and the conjectured bound $O(r^{1/2+\ve})$ remains hopelessly out of reach of the current technology. If we replace the circle by a hyperbola $xy=N$, with  $N\longrightarrow \infty$, the problem of counting lattice points with positive coordinates under the hyperbola  is known as Dirichlet's divisor problem and the situation here is completely analogous to the Gauss circle problem (see \cite[Chap.~1]{IK}). An interesting
lattice point counting problem  concerns counting elements $\gamma \in SL(2, \Z) $; i.e., integer matrices 
$\gamma=\begin{pmatrix}
a &b\\
c &d\\
\end{pmatrix}$
with determinant $ad-bc=1$, such that $||\gamma||\leq x$, where we take the norm to be the standard Euclidean norm on ${\R}^4$ after identifying the matrix $\gamma$ with the point $(a,b,c,d)$; i.e., 
\[
||\gamma||^2=a^2+b^2+c^2+d^2. 
\]
This problem can be interpreted as the analogue of the Gauss circle problem in the hyperbolic plane $\mathbb{H}$. To see why, let us define a distance function $d(z, w)$ between two points $z$ and $w$ in $\mathbb{H}$ by
\[
d(z,w)=\frac{|z-w|^2}{\Im z \Im w}+2,
\]
then $d(\gamma. i, i)=a^2+b^2+c^2+d^2$, where $\gamma.z$ denotes the usual M\"{o}bius action 
$
\gamma. z=(az+b)(cz+d)^{-1}. 
$
This distance is a minor modification of the  standard hyperbolic distance (see \cite[Chap. 1]{Iw2}). Thus the above  problem is equivalent to studying the asymptotic behaviour of the size of the set $\{\gamma \in \Gamma: d(\gamma.i, i)\leq x\}$, where
$\Gamma$ is the group $SL(2, \Z) $ and  one may also consider a more general problem by taking $\Gamma$ to be any discrete subgroups of $SL(2, \R) $ with finite 
co-volume. Unlike in the Euclidean case, even finding the asymptotic for  lattice point counting problems is not easy and the 
elementary counting arguments no longer work due to the negative curvature of the 
underlying space.
 Lattice point counting 
  problems are intimately related to the spectrum of the Laplace operator on  appropriate manifolds. For example, the Gauss circle problem is 
  nothing but the problem of improving the bound on the error term in the Weyl law for the flat torus ${\Z}^2\backslash{\R}^2$. The above hyperbolic lattice point counting  
  problem  is related to the spectrum of the hyperbolic Laplacian acting on $L^2(\Gamma\backslash\mathbb{H})$ (as a densely defined operator) and after initial advances made by Delsarte \cite{Delsarte} and Huber, who solved the problem for cocompact groups (see \cite{Hu1, Hu2}),
the original problem for $\Gamma=SL(2, \Z)$ could be tackled once Selberg established the complete
  spectral resolution of the Laplacian for  groups with parabolic elements.
Selberg's result (see, e.g., \cite[Corollary 15.12]{IK}) says
\be\label{Selberg}
\sum_{\substack{\gamma \in  SL(2, \Z)
\\ ||\gamma|| \leq x}} 1=6x^2+O(x^{4/3}).
\ee
 For an account of this  deep  and beautiful circle of ideas, see, e.g.,  \cite[Chap. 15]{IK}, \cite{Iw2}, \cite{Lax-Phillips},  \cite{Patterson} and the original articles by Selberg, viz. \cite{Selberg1}, \cite{Selberg2}. For an interesting twisted version of Selberg's result, see \cite{Good}.
Selberg's work has been vastly generalized by Duke, Rudnick and Sarnak \cite{DRS}
to  varieties $V$ with an algebraic group action for which $V(\mathbb{R})$ is affine and symmetric by applying advanced results from harmonic analysis. A special case of their result says
\be\label{DRS}
\sum_{\substack{\gamma \in  M(n, \Z)\\ \det(\gamma)=r\\
\\ \|\gamma\| \leq x}} 1=c_{n,r} x^{n^2-n}+O_{\ve}(x^{n(n-1)-1/(n+1)+\ve}),
\ee
where $ M(n, \Z)$ is the set of $n\times n$ integer matrices, $\|\cdot\|$ denotes the Euclidean norm $\sqrt{tr(\gamma {\gamma}^t)}$, $r\neq 0$ is a fixed integer,  (i.e., $r=O(1)$), $c_{n,r}>0$ is an explicit constant and
$\ve>0$ is arbitrary. The saving $\delta_n=1/(n+1)$ in the exponent of the error term in \eqref{DRS} was later  improved by Gorodnik, Nevo and Yehoshua \cite{GNY} to, roughly, $2\delta_n$ and very recently it  has been  improved further   by Blomer and Lutsko \cite{blomer2023hyperbolic} who have shown that the saving is at least $1+1/\sqrt{2}+O(1/n)$ (for $r=1$ and $ n\geq 3$). In the case $n=2$, however, the bound $O(X^{4/3})$ for the error term due to Selberg remains unbeaten to this day.

\noi
 The problem under discussion is sensitive to the choice of the norm  and all the works mentioned above use the fact that the Euclidean norm is invariant under the action of  $SO(2, \mathbb{R})$. 
Our goal here is to study the same counting problem  but after replacing the Euclidean norm by the supremum norm, i.e., we now take $\|\gamma\| =\|\gamma\|_{\infty}={\max}_{i,j}\{|a_{i,j}|\}$, where
$a_{i,j}$ runs over the entries of the matrix $\gamma$; i.e., we count integral matrices in an expanding box rather than an expanding ball. This choice of the norm seems more natural if we view our problem in the general framework of counting integral points on affine hypersurfaces defined over $\Z$ (see, e.g., \cite{HB}). However, unlike \cite{HB}, we do not count with smooth weights here and this makes the job significantly harder. We confine ourselves to the simplest case $n=2$ in this work and our emphasis is on obtaining an asymptotic formula with a strong bound on the error term.  The methods of the works mentioned above do not seem to be applicable here, at least directly, and even if those methods can be adapted to fit this set-up, the quality of the error term seems likely to be weak (see, e,g., the remarks following \cite[Cor.~2.3]{GN}).
 Nevertheless, drawing on deep results from the theory of automorphic forms and using delicate analytic tools we have been able to show
 the following.  
\begin{theorem}\label{main}
Suppose $\theta$ is any admissible exponent  towards the Ramanujan Conjecture. 
Let
\[
\mathcal{S}_r(X)=\sum_{\substack{\gamma \in  M(2, \Z)\\ \det(\gamma)=r\\
\\ ||\gamma||_{\infty} \leq X}} 1=\abcdsumX 1,
\]
where $a, b, c, d$ vary over the integers and $r$ is a fixed non-zero integer. 
Then we have 
\be\label{main1}
\mathcal{S}_r(X)=\frac{16}{\zeta(2)}\frac{\sigma(|r|)}{|r|}X^2+O_{\ve}(|r|^{\theta}X^{3/2+\ve}),
\ee
as long as $r=O\l(X^{{1/3}}\r)$.
\end{theorem}  
\begin{remark}\label{kim}
Conjecturally $\theta$ can be taken to be zero and by the work of Kim and Sarnak   \cite[Appendix 2]{KS} it is known that $\theta \leq \frac{7}{64}$. 
\end{remark} 
\noi
The analytic approach that we have followed  in proving the above result is quite flexible and  it allows for attaching suitable analytic or arithmetic weights  to one or more of the variables $a$, $b$, $c$ and $d$ and studying weighted sums of the form 
\[
{\mathop{\sum_a\sum_b\sum_c\sum_d}_{\substack{ad-bc=r\\|a|, |b|, |c|, |d|\leq X}}}f_1 (a) f_2(b)f_3(c)f_4(d),
\]
where, $f_i$'s can be, for example, functions on $\R$ with sufficient regularity conditions,  periodic arithmetical functions such as Dirichlet characters, additive characters or more generally a trace function.  One may also attach the divisor function or coefficients of automorphic forms to some of the variables. It is also possible to attach an arbitrary sequence of complex numbers to one of the variables as we show below. Such weighted sums  do occur naturally in practice, e.g., while studying amplified moments of $L$-functions (see, e.g., \cite{DFI-det}, \cite[\S 16]{DFI-Artin},  \cite{BC}).   We now state a few  results that demonstrate the wide applicability of the method. 
For brevity we only consider solutions to the determinant equations that are positive but extending the results to all integers is not difficult. 
\begin{theorem}\label{arbitrary}
Let $(\alpha(n))$ be a sequence  of complex numbers such that $\alpha(n)=O_{\ve}(n^{\ve})$ for some $\ve>0$. 
Then for any non-zero integer $r$, we have
\ba
\abcdsumrX \alpha(a) = A_{w,\alpha}(X,r) + O\l(\max\l(|r|, X^{5/3+\ve}\r)\r),\nonumber
\ea
where 
\ba 
A_{w,\alpha}(X,r) = \sum_{l|r}\mathop{\sum\sum}_{\substack{(a, c) =l\\c \neq 0 \\ a>1}} \frac {l \alpha(a)} {a} w(a)w(c) \int_1^X w\l(\frac{cx}{a}\r)\diff x, \nonumber
\ea
and $w$ is any function that satisfies the conditions \eqref{w0}-\eqref{w-prime}.
\end{theorem} 
\noi
In particular, by taking $\alpha(a)$ to be the constant sequence $1$ and following the method of extraction of the main term in Theorem \ref{main}, we obtain
 \[ 
\abcdsumrX 1=  \frac{2}{\zeta(2)}\frac{\sigma(|r|)}{|r|}X^2 + O_{\ve}\l(  \max \l(X^{5/3}, |r|\r)X^{\ve}\r).
\]
By symmetry of the solutions to the determinant equation, we obtain, therefore, the following corollary which increases the range of $r$ in Theorem \ref{main} at the cost of a worse error term. 
\begin{corollary} We have,
\[ 
\abcdsumX 1=  \frac{16}{\zeta(2)}\frac{\sigma(|r|)}{|r|}X^2 + O_{\ve}\l(  \max \l(X^{5/3}, |r|\r)X^{\ve}\r).
\]
\end{corollary} 
\noi
\begin{remark}
The above result has also been obtained independently by Muhammad Afifurrahman \cite{Afifur}  by a totally different method that uses clever counting arguments and builds up on earlier results on counting integer points on modular hyperbolas. \\ 
\end{remark}

\noi
We can generalize the above result by attaching  arbitrary periodic arithmetical functions  to two of the variables, say $b$ and $c$. Let $f_1$ and $f_2$ be two fixed periodic arithmetical functions with periods $q_1$ and $q_2$ respectively, say.  For simplicity, we take  $r=1$ and $q_1, q_2$ to be distinct primes.  Then we have the following result where we have used the notations of the previous theorem. 
\begin{theorem}\label{weighted}
We have, 
\ba
\abcdsumoneX \alpha(a)f_1(b)f_2(c) 
=\widehat{f_{1, q_1}}(0)\widehat{f_{2, q_2}}(0)A_{w,\alpha}(X, 1)
&+\prod_{j=1,2}\l(f_{j}(0)-\widehat{f_{j, q_j}}(0)\r)S(\alpha, w, q_1, q_2) \\
&+O_{ \alpha, f_1, f_2, \ve}(X^{5/3+\ve}), 
\ea
where, 
\[
S(\alpha, w, q_1, q_2)=\sum_{(a,q_1 q_2)=1} \frac{\phi(a)}{a^2}\sum_{j_1=1}^{\infty}\sum_{j_2=1}^{\infty}\frac{\alpha (a{q_1}^{j_1}{q_2}^{j_2})w(a{q_1}^{j_1}{q_2}^{j_2})}{{q_1}^{j_1+1}{q_2}^{j_2+1}}\int\int w(x)w(y)w\l(\frac{xy}{a{q_1}^{j_1}{q_2}^{j_2}}\r)\diff x \diff y,
\]
and 
\[
\widehat{f_{i, q_i}}(\beta)=\frac{1}{q_i}\sum_{\gamma (\text{mod }q_i)} f_i(\gamma)e(-\beta\gamma/q_i). 
\]
\end{theorem}

Taking $(\alpha(a))$ to be the constant sequence $1$, one obtains:
\begin{corollary}\label{special}
\[
\abcdsumoneX f_1(b)f_2(c) =\frac{2X^2}{\zeta(2)}\l(\widehat{f_{1, q_1}}(0)\widehat{f_{2, q_2}}(0) +\frac{\prod_{j=1,2}\l(f_{j}(0)-\widehat{f_{j, q_j}}(0)\r)}{q_1 q_2 ({q_1}^2 -1)({q_2}^2 -1)} \r)+O_{ f_1, f_2,\ve}(X^{5/3+\ve}).
\]
\end{corollary}
\noi
The reason for the weaker error terms in the above two theorems over Theorem \ref{main} is that whereas we exploit cancellations in a sum of Kloosterman sums that arises in the proof of Theorem \ref{main}, we are unable to do so in the proofs of the next two theorems since the sequence $(\alpha(a))$ is arbitrary.  If $\alpha(a)$ is a suitable sequence; e.g., a constant sequence, then it certainly possible to improve upon the error term, e.g., in Cor. \ref{special}, by appealing to the tools from the spectral theory of automorphic forms. \\

\noi
Finally, we mention a result from our companion article \cite{GG2} which is the exact analogue of Theorem \ref{main} where the variables are counted with smooth weights. As one would expect, we get a stronger error term. 
\begin{theorem}\label{smoothed} Let
$V$ is a fixed smooth function with compact support, say, in $[1,2]$ and let $r$ be a fixed  non-zero integer. Then we  have,
\be\label{smooth}
\abcdsumV V\l(\frac{a}{X}\r)V\l(\frac{b}{X}\r)V\l(\frac{c}{X}\r)V\l(\frac{d}{X}\r) = M_V(X,r) + O_{\ve}\l(| r|^{\theta}X^{1+\ve}\r) , \nonumber
\ee
where
\[
M_V(X,r)=\sum_{l|r_1}\sum_{k>0}\frac{\mu(k)}{k}\iiint \frac{1}{z} V\l(\frac{x}{X}\r)V\l(\frac{lky}{\ X}\r)V\l(\frac{lkz}{X}\r)V\l(\frac{r_1+lkxy}{ zlkX}\r) \diff x \diff y \diff z,
\]
and for $0\neq r\in (-3X^2, X^2)$, we may write $M_V(X, r) =c_V\frac{\sigma(|r|)}{|r|} X^2+ O(\sigma(|r|)\),
where $c_V>0$ is a constant depending only on $V$.
\end{theorem} 
\noi
The proof of this result follows the approach adopted here but it is considerably simpler than the proof of  Theorem \ref{main}. See \S \ref{proof-sketch} below or \cite{GG2} for further details. 

\vspace{-2mm}
 \begin{remark}
The reason we need to take $r=O(X^{1/3})$ in Theorem \ref{main}  is  explained in Remark \ref{1/3}. This restriction can  be relaxed with more work but we have not tried to do so in this paper. 
\end{remark} 

\begin{remark}
 Before the 2023 IISER Kolkata PhD thesis of the second named author in which Theorem \ref{main} was proved, the only other places where we could find this question addressed were  \cite{GN} and \cite{Roettger}. While  both these papers consider a far more general problem, but  when specialized to the case of  $2\times 2$ integer matrices, the results in these papers give much poorer savings in the exponent in the error term over the main term; namely,  $1/16$ in  \cite[Cor. 2.3]{GN}) and $1/15$ in  \cite[Thm. 1.2]{Roettger}. 
\end{remark}\vspace{-4mm}

\begin{remark}
Very recently, in September, 2025, Dhanda, Haynes and  Prasala \cite{DHP} and independently, Chapman and Mudgal \cite{Chapman-Mudgal}  have made a significant improvement to the error term in Theorem \ref{main} and now the error term is known to be $O(X^{1+\ve})$ which is essentially the best possible bound one can hope for (unless one considers removing the ``$\ve$") and moreover, this is proved in  the widest possible range of $r$ as well. The two proofs are different but the main idea is the same. It is an elegant idea that cleverly uses a symmetry of the solutions to the determinant equation that was not exploited by earlier authors. However, their idea uses a direct counting argument  instead of  analytic  techniques and therefore, if there are oscillatory  or smooth weights present, as in Theorems \ref{arbitrary}, \ref{weighted} or \ref{smoothed}, then it appears that this idea is no longer applicable. 
\end{remark}

\subsection{Sketch of the proofs}
\label{proof-sketch}
The choice of the norm here does not allow us to launch Harmonic Analysis directly into the counting problem. So we  ``forget" the  group structure of $SL(2, \Z)$ temporarily and treat the equation $\det(\gamma)=r$ as the equation defining a surface: $xy-zw=r$. Of course, the group structure remains present throughout, only hidden, and  after several transformations, it resurfaces as we are led to studying the automorphic spectrum of $SL(2, \Z)$ as explained below. \\
\noindent
 Let us first see what happens if we count with smooth weights, i.e., Theorem \ref{smoothed}. We interpret the determinant equation as a congruence modulo one of the variables and then apply the POisson summation formula. The zero-th frequency in the dual sum gives the main term and after one more application of Poisson summation we are left with, essentially, the following double sum:
\[
  \sum_{\substack{n\in \Z\\ n\neq 0}}\sum_{\substack{m\in \Z\\ m\neq 0}} \iint  V\l(\frac{x}{X}\r)V\l(\frac{y}{X}\r)\sum_{a} \frac{1}{a^2} V\l(\frac{a}{X}\r)V\l(\frac{r+yx}{aX}\r)
 e\l(\frac{-nx}{a}\r)e\l(\frac{-my}{a}\r) S(nr,-m, a)\diff x\, \diff y.
 \]
 At this stage, the Weil bound for Kloosterman sums already gives an asymptotic formula but we get a better error term by applying the Kuznetsov formula. 
Since the variable $a$ is weighted by a ``flat" function $V(x/X)$ (i.e., $y^j V^{(j)}(y)\ll 1$), it is not very difficult to obtain strong bounds on the Bessel transforms. See \cite{GG2} for details.\\
\noindent
  In contrast to the smooth case,  several  new features  emerge when we employ the  above strategy  to the sharp-cuts sum considered in Theorem \ref{main}.  By the Weil bound alone, it is not difficult to obtain an asymptotic formula with a weaker error term $O(X^{7/4+\ve})$ (see \eqref{sevenbyfour}).  A more elaborate  treatment of the exponential integrals yields the bound  $O(X^{5/3+\ve})$; see \eqref{5/3}. 
  However, our main goal is to obtain a stronger error term and to this end we employ the so called  ``Kloostermania" techniques (see \cite{Li}), i.e., we use the Kuznetsov formula (see Lemma \ref{kuz}) to transform the problem of estimating a sum of Kloosterman sums to that of estimating a sum of Fourier coefficients of Maass forms ranging over the full $GL(2)$ spectrum. While this approach is quite standard, the difficulty lies in obtaining a strong bound on the error term. Indeed,  if we apply the Kuznetsov formula directly after  the second application of Poisson summation, after smoothening the original sum,  then it seems extremely difficult to obtain  good estimates for the resulting Bessel transforms since 
 we no longer have the  ``flat" weight functions ( i.e. ones whose derivatives have  rapid decay).
 To overcome this problem we analyze the Fourier transform $\widehat{G}(n/a_1, m/a_1)$ (see \eqref{India}) carefully and extract stationary phases of the type 
  $e\l(-2\sqrt{mnX/a_1}\r)$  (see Prop. \ref{ET}). 
 However, a straight-forward application of the stationary phase lemma produces error terms that are  quite large. Instead, we imagine  the stationary point, which depends on three variables of summation $a_1$, $m,$ and $n$ (see \S \ref{B1} to \S \ref{conI4}), to be moving in an interval and we break up  the range of the stationary point into smaller subintervals and we also subdivide the range of integration suitably so that the stationary point is highly localized. See \S \ref{I5two} for details.
 When the stationary point is sufficiently localized,  the variables are  under severe restrictions and exploiting these restrictions leads to  strong estimates for the contributions of the error terms arising from the stationary phase analysis.  \\
  After we apply the Kuznetsov formula, the phase of the $K$-Bessel function combines with the already existing phase and in certain ranges of the variables we are able to extract a stationary phase which is essentially of the form $(mn)^{i\kappa_j}$ (see \eqref{lt}). At this stage, a large sieve inequality (see Lemma \ref{ls}) for the twisted coefficients $\rho_j(n)n^{i\kappa_j}$ that arose in a work of Deshouillers and Iwaniec  \cite{DINON} in connection with the Phillips-Sarnak theory  on destruction of cusp forms for generic  groups  (see \cite{PS}) gives us an extra saving in certain ranges of the variables.  In the complementary ranges, we obtain adequate saving by the Weyl law and the local (spectral) large sieve of Jutila (Lemma \ref{jutila}).  
  The extra saving one gets by exploiting the average over the variable $a$ can 
  be  compared  to what Heath-Brown calls ``double Kloosterman refinement" in \cite{HB}), though the method of averaging over the moduli  in \cite{HB} is quite different. \\
  \noi
  The proofs of the next two theorems follow, essentially, the same initial steps but we do not sum over the variable $a$ non-trivially. 
  \noi
We have been a bit brief  about some  details that are considered standard. The interested reader can find full details in the PhD thesis of the second author available at
    \url{https://sites.google.com/view/rachita-guria/research?authuser=0}.
\subsection{Notations and conventions}
 The letter $\ve$ denotes a positive real 
	number which can be taken to be as small as required and in different occurrences it may assume different values. 
	The notation ``$f(y) =O(g(y))$" or ``$f(y) \ll g(y)$", where $g$ is a positive function, means
	that there is a constant $c>0$ such that $|f(y)|\leq cg(y)$ for any $y$ in the concerned domain. 
	The dependence of this implied constant on some parameter(s) may sometimes
	be displayed by  a suffix (or suffixes) and may sometimes be suppressed. Also, $f(y)\asymp g(y)$ denotes $f(y)\ll g(y)$ as well as $g(y)\ll f(y)$. For a positive integer $n$, $\tau(n), \sigma(n), \phi(n)$ denotes, respectively the total number of positive divisors of $n$, the sum of the divisors of $n$, and the number of residue classes modulo $n$ that are prime to $n$. The symbol $\mathcal{L}(I)$ denotes the length of an interval $I$. The symbol 
	$\sim$, as in $\sum_{n\sim N}$ denotes $n$ is varying in the  interval $[N, 4N)$.
\subsection*{Acknowledgements}
The authors thank Valentin Blomer, Tim Browning,  Étienne Fouvry, Gergely Harcos,  Vinay Kumaraswamy, Gopal Maiti, Ritabrata Munshi, D. Surya Ramana,  Peter Sarnak, Prahlad Sharma and Igor Shparlinski for helpful comments. This work is part of the second author's PhD thesis. She thanks Indian Institute of Science Education and Research, Kolkata and all her teachers, especially Subrata Shyam Roy, for all the support. The authors also thank Indian Statistical Institute, Kolkata where this work was carried out for its excellent working atmosphere. 

\section{Preliminaries}
\subsection{Results from the theory of automorphic forms} Let us fix an orthonormal basis of Maass cusp forms $\{u_j(z)\}_1^{\infty}$ for the full modular group $  \text{SL}(2, \Z)$, consisting of  common eigenfunctions of all the Hecke operators $T_n, n\geq 1$ (see  \cite{DI} for definitions). Thus every $u_j(z)$ is an eigenfunction of the Laplace operator with the associated eigenvalue $s_j(1-s_j)$, $s_j \in \C$. We write $s_j$ as $s_j=\frac{1}{2}+i\kappa_j$ with $\kappa_j \in \C$. It is known that $\kappa_j \in \R$ and in fact $|\kappa_j| > 3.815$ (see \cite[Lemma 1.4]{motospectral}).
Every $u_j(z)$ has the following Fourier expansion at the cusp at $\infty$:
\be 
u_j(z) = \sqrt{y} \sum_{n \neq 0} \rho_j(n)K_{i\kappa_j}(2\pi |n|y)e(nx), \nonumber
\ee 
where $z = x+iy,$ $x,y \in \R$. Note that $K_{i\kappa}(y) = K_{-i\kappa}(y)$. Hence, without any loss of generality, we may assume that $\kappa_j >0$. 
Suppose $\lambda_j(n)$ is the eigenvalue of $u_j(z)$ for the Hecke operator $T_n$. 
\ba \label{heckerelation}
\lambda _j (n)\lambda _j (m) = \sum_{d |(m,n)} \lambda _j \l(\frac{mn}{d^2}\r) \text{ and }
\rho_j(\pm n) = \rho_j(\pm 1)\lambda_j(n).
\ea
\textbf{The Ramanujan–Petersson conjecture:} The Hecke eigenvalues at primes $p$ satisfy the bound 
\ba \label{ramanujan}
H(\theta): |\lambda_j(p)| \leq 2 p^{\theta},
\ea
where $\theta$ is a positive real number. It is conjectured that $H(\theta)$ holds for any $\theta>0$.
\begin{lemma}\label{kuz}(\textbf{The Kuznetsov formula: opposite sign case}) Let $f(t)$ be of $C^2$ class with compact support in $(0,\infty)$. Then for $n,m \geq 1$, we have 
\ba \sum_ c \frac{S(n,-m,c)}{c} f\l(\frac{4\pi \sqrt{mn}}{c}\r)&= \sum\limits_{j = 1}^{\infty} \rho_ j(n) \rho_ j(m)\check{f}(\kappa_ j) + \frac{1}{\pi}\int\limits_{-\infty}^{\infty} (nm)^{-i\eta}\sigma_{2i\eta}(n)\sigma_{2i\eta}(m)\frac{\cosh(\pi \eta)\check{f}(\eta)}{\l|\zeta (1+2i\eta)\r|^2}\diff \eta ,\nonumber
\ea 
where $\check{f}(\eta)$ is the Bessel transform 
\be \label{fcheckkuz}
\check{f}(\eta)= \frac{4}{\pi}\int\limits_ 0 ^{\infty} K_{2i\eta}(t) f(t) \frac{\diff t}{t}.  
\ee
\end{lemma}
\begin{proof}
See \cite[Eq. (1.20), Theorem 1]{DI1}.
\end{proof}

\begin{lemma}\label{weyl}(\textbf{The Weyl law})
We have,
\be 
N_{SL(2,\Z)}(K) = \l|\{j: |\kappa _j| \leq K \}\r| =  \frac{\textnormal{Vol}(SL(2,\Z) \backslash \h)}{4 \pi} K^2 + O(K\log K). \nonumber
\ee
\end{lemma}
\begin{proof}
See \cite[Eq. (11.5)]{Iw2}.
\end{proof}
\begin{lemma} \label{hec}
For any $\ve > 0$ and all $N \geq 1$, we have
\ba 
\sum_{n \leq N} |\lambda _j(n)|^2 \ll_{\ve} (\kappa_j)^{\ve}N \textnormal{ and } \frac{|\rho_j(1)|^2}{\cosh(\pi \kappa_j)} \ll_{\ve} \kappa_j ^{\ve}.\nonumber
\ea
\end{lemma}
\begin{proof}
See \cite[Lemma 1]{IW1} and \cite[Corollary 0.3]{HL}.
\end{proof}
\begin{lemma}\label{jutila}(\textbf{The local spectral large sieve inequality})
For $1 \leq \Delta \leq K$, we have 
\be 
\sum_{ K \leq \kappa_j \leq K+\Delta} \frac{|\rho_j(1)|^2}{\cosh(\pi \kappa _j )}\l|\sum_{n \leq N} a_n\lambda _j(n) \r| \ll (K\Delta + N)\|a\|^2(KN)^{\ve}.\nonumber
\ee
\end{lemma}
\begin{proof}
See \cite[Theorem 1.1]{J}.
\end{proof}

\begin{lemma}\label{ls}
For an arbitrary sequence of complex numbers $(a_n)_{n\leq N}$, we have
\be
\sum_{0 < \kappa_j \leq K }\frac{1}{\cosh(\pi\kappa_j)}
\l| \sum_{n \leq N}a_n\rho_j(n)n^{i\kappa_j}\r|^2 \ll (K^2+N^2)(NK)^{\ve}\sum_{n \leq N}|a_n|^2. \nonumber
\ee
\end{lemma} 
\begin{proof}
See  \cite[Theorem 6]{DINON}.
\end{proof} 
\noi
Finally, we recall the Weil bound for Klooterman sums (see \cite{IK} for definition).
\be \label{Weil}
\l|S(a,b;c) \r| \leq \tau(c)(a,b,c)^{\frac{1}{2}} c^{\frac{1}{2}}.  
\ee

\subsection{Results from Analysis}
\begin{lemma}\label{perron}
(\textbf{The Perron formula})\\
\be 
\frac{1}{2\pi i} \int _{c - iT}^{c+iT} \frac{x^s}{s}\diff s = h(x) + O\l( x^c\min \l(1, \frac{1}{T|\log x|}\r)\r), 
\textnormal{ where }
h(x) =
\begin{cases}
1, &\text{ if } x>1\\
\frac{1}{2}, &\text{ if } x=1\\
0, &\text{ if } x<1,
\end{cases} \nonumber 
\ee
for any $x > 0$, $x \neq 1$, and any $T , c> 0$, with an absolute implied constant. If $x = 1$, the same equality holds with the error term  replaced by $O(c/T)$. 
\end{lemma}

\begin{proof}
See \cite[Chap. 17]{Davenport}.
\end{proof}

\begin{lemma}\label{Stein}(\textbf{The $k$-th derivative bound})
Suppose $\psi \in C^1[a,b]$. Let $\phi$ be a real valued and smooth in $(a,b)$ such that $\l|\phi^{(k)}\r| \geq \lambda _k $ for all $x \in (a,b)$. Then
\be \label{k-der}
\l|\int\limits_a ^ b e^{i\phi(x)} \psi(x)\diff x \r| \leq c_k \lambda _k^{-1/k}\l[\|\psi\|_{\infty}+ \int\limits_ a ^b\l|\psi'(x)\r| \diff x\r]  
\ee holds when $k \geq 2$, or $k =1$ and $\phi'(x)$ is monotonic. Here $c_k>0$ is a constant that depends only on $k$. 
\end{lemma}
\begin{proof}
See \cite[Eq. (6) \& Prop. 2, Chap. 8]{S}.
\end{proof}

\begin{lemma}\label{stationary}(\textbf{The stationary phase lemma})
Let $g(x)$ be real-valued function on $[a,b]$ with $g'(x_0)=0$ for some $x_0 \in [a,b]$. Suppose $0 < \lambda_2 \leq g''(x) $ holds throughout the interval, and in addition $|g^{(3)}(x)| \leq \lambda_3$ and that $|g^{(4)}(x)| \leq \lambda_4$ for all $x \in [a,b]$. Then 
\be 
\int\limits_ a ^ b e(g(x))\diff x = \frac{e(g(x_0)+1/8)}{\sqrt{g''(x_0)}}+ O(R_1)+ O(R_2), \nonumber
\ee 
where 
\ba  R_1=\min\l(\frac{1}{\lambda_2 (x_0-a)},\frac{1}{\sqrt{\lambda_2}}\r)+\min\l(\frac{1}{\lambda_ 2(b-x_0)},\frac{1}{\sqrt{\lambda_2}}\r) \textnormal{ and } 
R_2 = \frac{(b-a)\lambda_4}{\lambda_ 2^2}+\frac{(b-a)\lambda_3^2}{\lambda_2^3}. \nonumber
\ea
If the bound $g''(x)\leq - \lambda_2 <0$ holds throughout the interval, then the main term becomes $
\frac{e(g(x_0)-1/8)}{\sqrt{|g''(x_0)|}}.$
\end{lemma}
\begin{proof}
See \cite[Theorem 9]{M}.
\end{proof}
\begin{lemma}\label{Huxley}(\textbf{The weighted stationary phase lemma})
Let $f(x)$ be a real function, four times continuously differentiable for $\alpha \leq x \leq \beta$, and let $g(x)$ be a real function, three times continuously differentiable for $\alpha \leq x \leq \beta$. Suppose that there are positive parameters $T, P, V, Q,$ with $
Q \geq \beta - \alpha, \quad V \geq Q/\sqrt{P}$, and positive constants $C_r$ such that, for $\alpha \leq x \leq \beta$,
\ba
\l|f^{(r)}(x)\r| \leq \frac{C_r P}{Q^r}, \quad \l|g^{(s)}(x)\r| \leq \frac{C_s T}{V^s} \quad \textnormal{and}\quad f''(x) \geq \frac{P}{C_2 Q^2}, \nonumber
\ea
for $r = 2,3,4$, and $s =0,1 2,3$. Suppose also that $f'(x)$ changes sign from negative to positive at a point $x = x_0$ with $ \alpha < x_0 < \beta$. If $P$ is sufficiently large in terms of $C_2,C_3$, and $C_4$, then we have
\ba 
&\int\limits_{\alpha} ^{\beta}g(x) e(f(x)) \diff x = \frac{g(x_0)e\l(f(x_0)+\frac{1}{8}\r)}{\sqrt{f''(x_0)}} + \frac{g(\beta)e(f(\beta))}{2 \pi i f'(\beta)} - \frac{g(\alpha)e(f(\alpha))}{2 \pi i f'(\alpha)}\\ 
&+ O\l(\frac{Q^4T}{P^2} \l(1+ \frac{Q}{V} \r)^2\l(\frac{1}{(x_0 - \alpha)^3}+ \frac{1}{( \beta - x_0 )^3} \r) \r)
+ O\l( \frac{QT}{P^{3/2}}\l(1+ \frac{Q}{V} \r)^2\r).\nonumber
\ea
\end{lemma}
\begin{proof}
See \cite[Lemma 5.5.6]{H}.
\end{proof}

\begin{lemma}\label{BKY1}
Let $Y \geq 1, T,Q,U,R > 0$ and suppose  $w$ is a smooth function with support in $[\alpha, \beta]$ satisfying $ 
w^{(j)}(t) \ll_ j TU^{-j} $. Suppose that $h$ is a smooth function on $[\alpha, \beta]$ such that 
$ \l|h'(t)\r|\geq R,\quad \textnormal{ and }\quad h^{(j)}(t) \ll_ j YQ^{-j}, \textnormal{ for } j = 2,3 \dots$, for some $R >0$.
Then the integral $I$ defined by 
\be \label{integralstationary}
I = \int\limits_{\R} w(t)e^{ih(t)} dt \quad \textnormal{ satisfies }\quad
I \ll_ A (\beta - \alpha)T\l[\l(\frac{QR}{\sqrt{Y}}\r)^{-A}+ (RU)^{-A}\r].
 \ee
\end{lemma}
\begin{proof}
See \cite[Lemma 8.1]{BKY}.
\end{proof}

\begin{lemma}\label{BKY2}
Let $0 < \delta <\frac{1}{10},T,Y,V,V_1,Q >0,$ $Z= Q+X+Y+V_1+1$ and assume that 
\be \label{bound}
Y \geq Z^{3\delta},\quad\quad V_1\geq V \geq \frac{QZ ^{\frac{\delta}{2}}}{\sqrt{Y}}. 
\ee 
Suppose $w$ is a smooth function on $\R$ with support on an interval $J$ of length $V_1$ and  satisfies  $w^{(j)}\ll_ j T V^{-j}$, for all $j \in \N_0$, and $h$ is a smooth function on $J$ such that there exists a unique point $t_0 \in J$ with  $h'(t_0)=0$, and suppose, furthermore, that 
$
h''(t) \gg YQ^{-2}, h^{(j)}(t)\ll_ j YQ^{-j}, \text{ for } j = 1,2,3,\dots,t \in J.
$
Suppose also that $h'(t)$ changes sign from negative to positive at a point $t = t_0$. Then the integral $I$ defined in \eqref{integralstationary} has an asymptotic expansion of the form
\be \label{BKY3}
I = \frac{e^{ih(t_0)}}{\sqrt{h''(t_0)}}\sum\limits_{n \leq 3\delta^{-1}A} p_n(t_0)+ O_{A,\delta}\l(Z^{-A}\r),\ee where
\be \label{BKY4}
p_n(t_0) = \frac{\sqrt{2\pi}e ^{\frac{\pi i}{4}}}{n!}
\l(\frac{i}{2h''(t_0)}\r)^n G^{(2n)}(t_0),\text{ } G(t) = w(t)e^{iH(t)}, \textnormal{ and } H(t)= h(t)- h(t_0)-\frac{h''(t_0)}{2} (t-t_0)^2,
\ee 
for some arbitrary $A > 0$. Furthermore, each $p_n$ is a rational function in $h'', h''',\dots,$ satisfying 
\be \label{derivative} 
\frac{d^j}{dt_0^j}p_n(t_0) \ll_{j,n}T \l(V^{-j}+Q^{-j}\r)\l(\l(\frac{V^2Y}{Q^2}\r)^{-n} + Y^{-\frac{n}{3}}\r).
\ee
\end{lemma}
\begin{proof}
See \cite[Prop. 8.2]{BKY}.
\end{proof}

\section{Initial reduction and the main steps of the proof}\label{reduction}
To prove Theorem \ref{main}, it is enough to prove that for $0<|r|\ll X^{1/3}$,
\be \label{toshow}
S_r(X) = \mathop{\sum\sum\sum\sum}_{\substack{1 \leq a,b,c,d \leq X \\ ad- bc =r}}1
=\frac{2}{\zeta(2)}\frac{\sigma(|r|)}{|r|}X^2+O(r^{\theta}X^{3/2+\ve}),
\ee
since $\mathcal{S}_r(X)=8S_r(X)$ by symmetry. The condition on $r$ will be required much later and some of the initial propositions hold for any non-zero integer $r$. Interpreting the equation  as a congruence modulo $a$, we write
\[
S_r(X) = \mathop{\sum\sum\sum}_{\substack{1 \leq a,b,c \leq X \\ bc \equiv - r \moda \\ a-r \leq bc \leq aX-r}}1
 = \mathop{\sum\sum\sum}_{\substack{1 \leq a,b,c \leq X \\ bc \equiv - r \moda \\ a \leq bc \leq aX}}
+ O\l((|r|+X)X^{\ve}\r); \text{  i.e.,}
\]
\be\label{SrX}
 S_r(X) =  S_r^{(0)}(X)+ O\l((|r|+X)X^{\ve}\r),
 \ee
 say. 
We approximate $S_r^{(0)}(X)$ by a weighted version. Let $H$ be a parameter to be fixed later subject to the condition 
\be\label{H-size}
\sqrt{X}\leq H\leq X.
\ee
Let $w:\R\to \R$ be a smooth function satisfying the  conditions
\begin{align}
&(1)  \ w(x)=0\textnormal{ if } x\not\in (1-H, X+H),\label{w0}\\
& (2) \  w(x)=1\textnormal{ if } x\in [1, X], \label{w1}\\
& (3) \ 0<w(x)<1 \textnormal{ if } x\in (1-H, 1)\cup (X, X+H),\label{w-support}\\
& (4)\ w^{(j)}(x) \ll_j H^{-j}\textnormal{ for every } j\geq 1 \label{w-prime}.
\end{align}
Then, we clearly have 
\be\label{Sr0X}
S_r^{(0)}(X)=S_w(X, r)+O\l(HX^{1+\ve}\r),
\ee
where
\be\label{defSrwX}
S_w(X, r)= \mathop{\sum\sum\sum}_{\substack{ bc \equiv - r \moda \\ c \neq 0 \\ a>0}}  w(a)w(b)w(c)w\l(\frac{bc}{a}\r).
\ee
Let us define $l = (a,c), r_1=r/l, c_1=c/l, a_1=a/l$ and for any integer $a\neq 0$ and any integer $c$, we define
\be \label{F(x)}
F_{a,c}(x):=w(x)w\l(\frac{cx}{a}\r). 
\ee
Applying Poisson summation to the $b$-sum and separating the zero frequency, we obtain
\ba\label{SW-decomp}
S_w(X, r) &=  \sum_{l|r}\mathop{\sum\sum}_{\substack{(a_1,c_1) = 1\\c_1\neq 0 \\ a_1 >1}} \frac {w(la_1)} {a_1} w(lc_1)\widehat{F_{a_1,c_1}}\left(0\right)+\sum_{l|r}\mathop{\sum\sum}_{\substack{(a_1,c_1) = 1\\c_1\neq 0 \\ a_1 >1}}\frac {w(la_1)} {a_1} w(lc_1) \sum_{\substack{n\in \Z\\ n\neq 0}}e\left(\frac{-nr_1\Bar{c_1}}{a_1}\right)\widehat{F_{a_1,c_1}}\left(\frac{n}{a_1}\right)\\
& =A_w(X, r)+B_w(X, r),
\ea
say. 
\begin{proposition} \label{MT}
We have,
\be \label{awx}
A_w(X, r)= \frac{2}{\zeta(2)}\frac{\sigma(|r|)}{|r|}X^2+ O\l(\frac{\sigma(|r|)}{|r|}HX^{1+\ve}\r).
\ee
\end{proposition}
 This will be proved in the next section. The following intermediate proposition ``prepares" $B_w(X, r)$ for an application of the Kuznetsov formula. First we  introduce the notations:
\ba \label{uniformphase}
\mathop{\sum_{l, m,n}}:= \sum_{l|r}\mathop{\sum_{1 \leq m \ll \frac{X^{1+\ve}}{H}}\sum_{1 \leq n \ll \frac{X^{1+\ve}}{lH}}}, \textnormal{ and }\mathcal{H}(t):=\mathcal{H}(m, n,t) = -2 \sqrt{\frac{mnt}{a_1}}-\frac{1}{8}.
\ea 
\begin{proposition}\label{ET} (\textbf{The first main proposition})
We have,
\ba 
&B_w(X, r) =  \frac{ X^{\frac{1}{4}}}{2\sqrt{2}\pi i} \mathop{\sum_{l, m,n}}_{ nl \leq m} \sum_ {a_1 >0}  \frac{ S(nr_1, -m, a_1)}{a_1}\l(\frac{a_1}{mn}\r)^{\frac{3}{4}}
\l(e\l( \mathcal{H}(X) \r) - e\l( - \mathcal{H}(X) \r)\r)  w \l(\frac{ma_1}{n}\r) \\
& +   \frac{ X^{\frac{1}{4}}}{2\sqrt{2}\pi i} \mathop{\sum_{l, m,n}} _{m < nl} \sum_ {a_1 >0}  \frac{ S(nr_1, -m, a_1)}{a_1}\l(\frac{a_1}{mn}\r)^{\frac{3}{4}}
\l(e\l( \mathcal{H}(X) \r) - e\l( - \mathcal{H}(X) \r)\r) w\l(\frac{nl^2a_1}{m}\r) \\
& + \frac 1{\sqrt{2}} \int \limits_{X} ^{ X+ H} \frac{w(u)}{u^{1/4}}  \mathop{\sum_{l, m,n}}_{ nlX \leq mu}  \sum_ { a_1 > 0}  \frac{ S(nr_1, -m, a_1)}{a_1} \l(\frac{a_1}{mn}\r)^{\frac{1}{4}} 
 \l( e\l( \mathcal{H}(u) \r) + e\l( -\mathcal{H}(u) \r)\r) w\l( \frac{ma_1u}{nX}\r)\\
& + \frac 1{\sqrt{2}} \int \limits_{X} ^{ X+ H} \frac{w(u)}{u^{1/4}}
  \sum_{l |r} \mathop{\sum_{l, m,n}}_{mu < nlX}  \sum_ { a_1 >0 }  \frac{S(nr_1, -m, a_1)}{a_1} \l(\frac{a_1}{mn}\r)^{\frac{1}{4}} 
 \l( e\l( \mathcal{H}(u) \r) + e\l( -\mathcal{H}(u) \r)\r) w\l( \frac{nl^2a_1X}{mu}\r)\\
& + O\l(X^{\ve}\l(X^{\frac{3}{2}} + X^{\frac{5}{4}}\sqrt{\frac{X}{H}}+ \frac{X^2}{H} + \l(\frac{X}{H}\r)^{3} + \frac{X^{\frac{5}{2}}}{H^2} \r)\r).
\nonumber
\ea
We express the right hand side more simply as
\[
 \l( R_{1,w}(X)+  R'_{1,w}(X) \r)+ \l( R_{2,w}(X)+  R'_{2,w}(X) \r)+ \l( R_{3,w}(X)+  R'_{3,w}(X) \r)+\l( R_{4,w}(X)+  R'_{4,w}(X) \r)+ O(\dots),
\]
\end{proposition}
\begin{remark}\label{5/3}
An application of  \eqref{Weil} at this stage yields 
\ba  \nonumber
B_w(X, r) \ll  X^{\ve}\l(\frac{X^{2}}{\sqrt{H}}+ X^{\frac{3}{2}}+ X^{\frac{5}{4}}\sqrt{\frac{X}{H}} + \frac{X^2}{H}+ \l(\frac{X}{H}\r)^{3} + \frac{X^{\frac{5}{2}}}{H^2} \r).
\ea
Now, from \eqref{SrX}, \eqref{SW-decomp}, and Prop. \ref{MT} with the 
optimal choice for $H = X^{2/3}$,  we obtain
\ba \nonumber
{S_r} (X)=  \frac{2}{\zeta(2)}\frac{\sigma(|r|)}{|r|}X^2 + O\l(  \max \l(X^{5/3}, |r|\r)X^{\ve}\r).
\ea
\end{remark}
\noindent
The next step is to exploit cancellations in the sum of Kloosterman sums using tools from the spectral theory of automorphic forms and the outcome is the next proposition.
\begin{proposition}\label{second prop}(\textbf{The second main proposition})
We have, for $0\neq r\ll X^{1/3}$,
\ba
B_w(X, r) &\ll X^{\ve} \l[ X^{\frac{3}{2}} + X^{\frac{5}{4}}\sqrt{\frac{X}{H}}+ \frac{X^2}{H} + \l(\frac{X}{H}\r)^{3}+ \frac{X^{\frac{5}{2}}}{H^2}  
+ r^{\theta}  \l(X\l(\frac{X}{H}\r) + X^{\frac{1}{2}} \l( \frac{X}{H}\r)^2 + X^{\frac{1}{3}+\frac{1}{4}}\l( \frac{X}{H}\r)^{\frac{3}{2}} \r) \r] . \nonumber
\ea
In particular, for the choice $H=\sqrt{X}$, 
\[
B_w(X, r)\ll r^{\theta}X^{\frac{3}{2}+\ve}. 
\]
\end{proposition}
\subsection{Proof of Theorem \ref{main}}
The theorem follows from Prop. \ref{MT} and Prop. \ref{second prop}, taking into account \eqref{SW-decomp}, and \eqref{SrX} after making the  choice $H=\sqrt{X}$. 

\section{The Main term and The Proof of Proposition \ref{MT}}\label{Mainterm}
The assumptions on the function $w$ allow us to write
\begin{align*}
A_w(X, r) &= \sum_{l|r}\mathop{\sum\sum}_{\substack{(a_1,c_1) = 1\\a_1c_1\neq 0 \\ a_1>0}} \frac {1}{a_1} w(la_1)w(lc_1) \int\limits _1^X w\l(\frac{c_1x}{a_1}\r)\diff  x
+O\l(H\sum_{l|r}\mathop{\sum\sum}_{\substack{(a_1,c_1) = 1\\a_1c_1\neq 0 \\ a_1>0}} \frac {1}{a_1} w(la_1)w(lc_1)\r),
\end{align*}
where the $O$-term is easily seen to be $O\l(\frac{\sigma(|r|)}{|r|}HX^{1+\ve}\r)$; and as for the first sum, we use \eqref{w1} to decompose  $A_w(X, r)$ as 
\begin{align*}
A_w(X, r) &=  \sum_{l|r}\sum_{\frac{1}{l} \leq a_1 \leq \frac{X}{l}} \frac {1}{a_1} \sum _{\substack{a_1 \leq c_1 \leq \frac{X}{l} \\ (a_1,c_1) = 1  }}  \int\limits_1^{\frac{a_1X}{c_1}}1 \diff x
+\sum_{l|r}\sum_{\frac{1}{l} \leq a_1 \leq \frac{X}{l}} \frac {1}{a_1} \sum _{\substack{\frac{1}{l} \leq c_1 \leq a_1 \\ (a_1,c_1) = 1}}\int\limits_{\frac{a_1}{c_1}}^X 1 \diff x 
+  O\l(\frac{\sigma(|r|)}{|r|}HX^{1+\ve}\r)\\
&=A_{1,w} (X)+ A_{3,w} (X) + O\l(\frac{\sigma(|r|)}{|r|}HX^{1+\ve}\r),
\end{align*}
say,
where the contribution of the terms with $a_1>X/l$ is absorbed in the error term. 
Now, we obtain
\ba
A_{1,w} (X) & = \sum_{l|r} \sum_{\frac{1}{l} \leq a_1 \leq \frac{X}{l}} \frac {1}{a_1} \sum _{\substack{a_1 \leq c_1 \leq \frac{X}{l} \\ (a_1,c_1) = 1}} \l(\frac{a_1X}{c_1}-1\r) =\frac{\sigma(|r|)}{|r|}\frac{X^2}{\zeta(2)}+ O\l(\frac{\sigma(|r|)}{|r|}X (\log X)^2 \r), \nonumber
\ea
by executing the integral and removing the coprimality condition by M\"{o}bius inversion;  and  $A_{3,w} (X) $ satisfies the same formula.

\section{The error term: Initial Transformations}\label{B0} 
Switching the order of summation and applying Poisson summation to the sum over $c_1$ after subdivision into residue classes modulo $a_1$, we get 
\ba\label{India}
B_w(X, r)&= \sum_{l|r} \sum_{a_1>1}\frac{w(la_1)}{a_1^2} \sum_{\substack{n\in \Z\\ n\neq 0}}\sum_{m \in \Z}
 S(nr_1, -m , a_1) \widehat{G}(n/a_1, m/a_1),
 \ea
 where $G$ is defined by 
 \[
 G(x, y)=w(x)w(ly)w\l(\frac{xy}{a_1}\r).
 \]
The decay of  $\widehat{G}$ limits the effective ranges of the variables $m$ and $n$  to $ \frac{X^{1+\ve}}{H}$ and $\frac{X^{1+\ve}}{lH}$, respectively, as shown in the lemma below. First recall the definition in \eqref{F(x)}. 
\begin{lemma}
We have
\be \label{basic}
\widehat{F_{a, c}}(y)\ll_k H^{1-k} y^{-k}\l(1+\left|\frac{c}{a}\right|^{k-1}\r),
\ee
for every integer $k\geq 1$. Furthermore. we have the bound
\ba \label{hatG}
\widehat{G}(u, v) \ll \min \l\{\frac X v, \frac{X}{lu}\r\},
\ea
where $u$ or $v$ are any two non-zero real numbers.
\end{lemma}

\begin{proof}
Using \eqref{F(x)} we integrate by parts $k$ times, getting
\begin{align*}
\widehat{F_{a, c}}(y) = \frac{1}{(-2\pi iy)^k}\sum_{j=0}^k \int\limits _{-\infty}^{\infty} w^{(j)}(x){\l(\frac{c}{a}\r)}^{k-j}w^{(k-j)}\l(\frac{cx}{a}\r)e(-xy)\diff x .
\end{align*}
Note that $w^{(k)}\l(\frac{cx}{a}\r)=0$ unless $\frac{cx}{a} \in (1-H, 1)\cup (X, X+H)$, and $w^{(k)}\l(\frac{cx}{a}\r)$ vanishes outside two intervals of length $O\l(\l|\frac a c\r| H\r)$. By \eqref{w-prime}, we obtain the bound $O({\l|\frac {c}{a}\r|}^{k-1}y^{-k}H^{1-k})$ for the term corresponding to $j=0$. 
The other summands together  contribute only $O(y^{-k}H^{1-k})$ by \eqref{w0}, \eqref{w1}, \eqref{w-support} and \eqref{w-prime}.
Finally, by  applying the bound \eqref{basic} for $k =1$ in one of the integrals  and by  trivially bounding the other integral using the support of the function $w$, we obtain the bound \eqref{hatG}.
\end{proof}
\noi
We isolate the term $m=0$ for which the Kloosterman sum becomes a Ramanujan sum. Using the standard bound \(|r_q(n)|\leq (q, n)\) for Ramanujan sums
(see \cite[Eq. (3.5)]{IK}), we  easily obtain the following: 
\be\label{BW-decomp}
B_w(X, r)=   B_{1,w} (X) +O(X^{1+\ve}),
\ee
where $B_{1,w} (X)$ denotes the sum over all $m\neq 0$.

\subsection{Analysis of $B_{1, w} (X)$: initial steps}\label{B1}

At this stage, \eqref{Weil} yields $B_{1,w}(X)\ll \tau(r)X^{5/2+\ve}/H$ and one obtains the asymptotic formula (with the optimal choice $H=X^{3/4}$)
\be\label{sevenbyfour}
{S_r} (X)=  \frac{2}{\zeta(2)}\frac{\sigma(|r|)}{|r|}X^2 + O\l( X^{7/4+\ve}\r).
\ee
Our goal is to improve upon the size of the error term in terms of $X$. First we 
note that the contribution of the  terms with $la_1>X$  is quite small as the next proposition shows. Let
 \ba
\mathcal{E}(Y) &:=\sum_{l|r} \sum_{0 \neq n \ll \frac{X^{1+\ve}}{lH}} \sum_{0 \neq m \ll \frac{X^{1+\ve}}{H}} \sum_{\substack{ a_1>0 \\ la_1 \in (Y,Y+H] }}\frac{w(la_1)S(nr_1, -m, a_1)}{a_1^2} \widehat{G}(n/a_1, m/a_1). \nonumber
\ea
\begin{proposition}\label{H-interval}
If $Y\asymp X$, then
\ba \nonumber
\mathcal{E}(Y) \ll  X^{\frac{3}{2}+\ve} .
\ea
\end{proposition}
\begin{proof}
This follows easily from \eqref{hatG} and \eqref{Weil}.
\end{proof}
\begin{remark}\label{largea1}
Thus, wherever we have a sum over $a_1$, we can disregard the part $\sum_{la_1>X}$ by adding the term $O\l(  X^{\frac{3}{2}+\ve} \r)$. We do not display it every time since in the final result we have this error term.
\end{remark}
We now state a result that will be applied several times. It follows easily from \eqref{Weil}.
\begin{proposition}\label{10.1}
Let $l \geq 1$, and $\alpha, \beta, \gamma \in \R$, such that $ \alpha \leq \frac{3}{2}$ and $\beta, \gamma \leq 1$. For $0 < U \leq X $ we have
\ba 
\Delta(\alpha,\beta,\gamma,l,U)&:=  \sum_{\substack{ 0 \neq n\ll \frac{X^{1+\ve}}{lH} }}\frac{1}{n^{\gamma}}\sum_{\substack{ 0\neq  m\ll \frac{X^{1+\ve}}{H}}}\frac{1}{m^{\beta}}
\sum_{\substack{1 \leq la_1 \leq U }}\frac{w(la_1)\l|S(nr_1, -m, a_1)\r|}{a_1^{\alpha}}
 \ll  U^{-\alpha+\frac{3}{2}+\ve} \l(\frac{X}{H}\r)^{-\gamma -\beta +2}  .\nonumber
\ea
\end{proposition}
\subsection{Decomposition of $\widehat{G}(n/a_1, m/a_1)$}
We have by \eqref{w1},
\ba
&\widehat{G}(n/a_1, m/a_1)  =\l( \int\limits_1 ^X   1+ \int\limits_ {1-H}^1 +\int\limits_ X ^{X+H} w(x) \r) \int w(ly)w\l(\frac{yx}{a_1}\r)  e\left(\frac{-my}{a_1}\right) e\l(\frac{-nx}{a_1}\r) \,\diff y\, \diff x \nonumber
\ea
By integration by parts in the $y$-integral followed by a trivial estimation of the $x$-integral the last two terms are $O(a_1 H /m)$.
Now we decompose the first integral  using   \eqref{w0} and \eqref{w1} as: 
\ba \label{hatGintegral}
&\widehat{G}(n/a_1, m/a_1) =\int\limits_{1} ^{la_1} \int\limits_{\frac{a_1}{x}}^{\frac{X}{l}} e\l(\frac{-nx-my}{a_1}\r)\diff y \diff x 
+ \int\limits_ {1} ^{la_1} \l[\int\limits_ {\frac{X}{l}} ^{\frac{X+H}{l}}+\int\limits_ {\frac{a_1}{x}(1-H)} ^{\frac{a_1}{x}}\r] w(ly)w\l(\frac{xy}{a_1}\r) e\l(\frac{-nx-my}{a_1}\r) \diff y \diff x \\
&+\int\limits_ {la_1} ^{X} \int\limits^{\frac{a_1X}{x}}_{\frac{1}{l}} e\l(\frac{-nx-my}{a_1}\r) \diff y \diff x 
+ \int\limits_ {la_1} ^{X}\l[ \int\limits_{\frac{a_1X}{x}} ^{\frac{a_1}{x}(X+H)} +\int\limits_ {\frac{1}{l}(1-H)} ^{\frac{1}{l}} \r]w(ly)w\l(\frac{xy}{a_1}\r) e\l(\frac{-nx-my}{a_1}\r) \diff y \diff x +O\l(\frac{a_1H}{m}\r)\\
&=I_1+I_2+I_3+I_4+I_5+I_6+O(a_1H/m),
\ea
say. Accordingly, we decompose $B_{1, w}(X)$ as 
\[
B_{1, w}(X)=\sum_{j=1}^6 B_{1, w}(X, I_j)+O\l(X^{3/2+\ve}\r),
\]
where we use Prop. \ref{10.1} to estimate the contribution of the error term. 
From $I_4$ and $I_5$ we  extract stationary phases for further analysis. The other terms can be handled by the Weil bound and they contribute only to the error term.

\section{Contributions of $I_1, I_2, I_3,$ and $I_6$}\label{B2}
\begin{proposition}\label{I2,I3,I6}
We have,
\ba \nonumber
B_{1, w}(X, I_1), B_{1, w}(X, I_2), B_{1, w}(X, I_3), B_{1, w}(X, I_6)\ll  \l(X^{\frac{5}{4}}(X/H)^{1/2}+ (X/H)^{3}+ X^{\frac{3}{2}}+X^{\frac{5}{2}}/H^2\r)X^{\ve}
\ea
\end{proposition}
\begin{proof} 
It is enough to bound the integrals as the outside sums are estimated after an application of \eqref{Weil}. We first consider $I_2$; and $I_6$ can be handled in a similar manner.  
Interchanging the order of the integrals appearing in $I_2$ and using integration by parts in the $x$-integral, we get $I_2   \ll a_1H/nl $ by the properties the function $w$. 
Now we consider $I_3$. From \eqref{hatGintegral}, we have
\ba \nonumber
I_3 &= \int\limits_ {1} ^{la_1} e\l(\frac{-nx}{a_1}\r) \int\limits_ {\frac{a_1}{x}(1-H)} ^{\frac{a_1}{x}} w(ly)w\l(\frac{xy}{a_1}\r) e\l(\frac{-my}{a_1}\r)\diff y \diff x 
= \int\limits_ {1} ^{la_1}  \l[\int\limits_ {\frac{a_1}{x}(1-H)} ^{0} + \int\limits_ {0}^{1} +
\int\limits_{1} ^{\frac{a_1}{x}} \r]. 
\ea
In the resulting three double-integrals, the last two can be bounded trivially and they contribute $O(a_1 \log X)$. 
In the first one, let us denote it by ${I_3}^{\ast}$, we interchange the $x$ and the $y$ integrals and the change of variable $y\rightarrow -y$ yields
\ba \nonumber
{I_3}^{\ast}&:=\int\limits_ { 0}^{a_1(H-1)} w(-ly) e\l(\frac{my}{a_1}\r) \int\limits_ {1} ^{\min\{la_1, a_1 (H-1)/y)\}} w\l(\frac{-xy}{a_1}\r) e\l(\frac{- nx}{a_1}\r)\diff x \diff y \\
&=\int\limits_ { 0}^{(H-1)/l} \int\limits_ {1} ^{la_1} +\int\limits_ { (H-1)/l}^{a_1(H-1)} \int\limits_ {1} ^{a_1 (H-1)/y} 
={I_{3, 1}}^{\ast}+{I_{3, 2}}^{\ast},
\ea
say. 
By integration by parts in the $x$-variable followed by trivial estimation of the $y$-integral, we have ${I_{3, 1}}^{\ast}\ll \frac{a_1H}{nl}$. For ${I_{3, 2}}^{\ast}$, we note that from the support of the function $w$ (see \eqref{w0}), $w(-ly)=0$ if $y>(H-1)/l$. Hence, ${I_{3, 2}}^{\ast} = 0$.
Combining the above bounds, 
$
I_3 \ll a_1 \l(\frac{H}{nl} +  \log X \r).
$ 
Next we consider $ I_1$. We execute the integration over  $y$ first. Thus,
\ba \label{I_1} 
I_1 = \frac{1}{2\pi i}\frac{a_1}{m}  I'_{m,n,a_1} +  O\l(\frac{a_1H}{m}+ \frac{a_1^2}{mn}\r), \text{ where } I'_{m,n,a_1} = \int\limits_ 1 ^ {la_1} e(h_1(x)) \diff x, \text{ and }
h_1(x)= -\frac{nx}{a_1}-\frac{m}{x}. \nonumber
\ea
The contribution of the $O$-terms is $O(X^{3/2+\ve})$ by Prop. \ref{10.1}. 
We  now analyze the exponential integral. If $mn<0$,
 there is no stationary point. Hence, by \eqref{k-der}, 
$
I'_{m,n,a_1}  \ll a_1/n.
$
If $mn>0$,  there is a stationary point given by \(x_0 = \sqrt{\frac{a_1m}{n}}\). Let us assume further $m, n >0$ since the analysis in the other case is similar and we get the same bound.
We consider two cases. \\
\textbf{Case I}: $x_0\in [2, 2la_1/3]$.
 Here we split the interval $[1, la_1]$ in three parts as follows:
\be \nonumber
I'_{m,n,a_1} = \l(\int_1 ^{x_0 - \delta x_0} + \int_{x_0 - \delta x_0} ^{x_0 + \delta x_0} +  \int_{x_0 + \delta x_0} ^{la_1} \r) e (h_1(x))\diff x=J_1+J_2+J_3, 
\ee
say, where $\delta=1/2$. 
By \eqref{k-der} with $k=1, 2$ respectively, we get \(J_1, J_3 \ll a_1/n\) and
\(J_2 \ll 
a_1^{3/4}m^{1/4}/ n^{3/4}\). 
\textbf{Case II}: $x_0\notin [2, 2la_1/3]$.
In this case, the required bound follows after applying \eqref{k-der} with $k=2$ to $I'_{m,n,a_1}$ and \eqref{Weil} once we observe that the conditions $x_0<2$ and $x_0>2la_1/3$ correspond to the conditions
$a_1 \leq 4n/ m$ and  $a_1 \leq 9m/4nl^2$, respectively. 
\end{proof}
\section{Contributions of $I_4$ and $I_5$} \label{conI4}
\subsection{Contributions of $I_4$}\label{I5one} 
Here we shall assume that $a_1$ satisfies $la_1\leq X-1$ and this is justified by Remark \ref{largea1}. Analogously to \eqref{I_1},  here we have the expression
\ba \nonumber
I_4 = \frac{1}{2\pi i}\frac{a_1}{m} I_{m,n,a_1}   + O\l(\frac{a_1H}{m} + \frac{a_1^2}{mn}\r), \textnormal{ where } I_{m,n,a_1} :=\int _{la_1}^X e(h(x))\, \diff x, \textnormal{ and } h(x):= -\frac{nx}{a_1}-\frac{mX}{x} . \nonumber
\ea
By Prop. \ref{10.1},  the contribution of the $O$-terms is absorbed into the error term. Differentiating $h(x)$, 
\be  \label{secondderi}
h'(x)= -\frac{n}{a_1}+\frac{mX}{x^2},\quad \text{and} \quad h''(x) = - \frac{2mX}{x^3}. 
\ee
For $mn < 0$, there is no stationary point and $I_{m,n,a_1}  \ll a_1/n$ by \eqref{k-der} and the contribution of this is sufficiently small by Prop. \ref{10.1}. The case left to consider is  $mn>0$ and we shall assume that both $m$ and $n$ are positive since the complementary case will be similar. In this case the stationary point is given by 
\be \label{stationarypoint1}
x_0 = x_0(m,n,a_1) := \sqrt{\frac{ma_1X}{n}} \ll X^{5/4} 
\ee 
Now, if we give an upper bound for exponential integrals as in the case of $I_1$ then the error term is not strong enough. We, therefore, refine our analysis by localizing the stationary point $x_0$ and extracting  stationary phases and estimating the ensuing error terms carefully. This leads to the following proposition, where we use the notation in \eqref{uniformphase}. 
\begin{proposition} \label{G}
We have, 
\ba \nonumber
B_{1, w}(X, I_4) &= \frac{X^{\frac{1}{4}}}{2\sqrt{2}\pi i}\mathop{\sum_l\sum_m\sum_n\sum_{a_1 >0}}_{ \sqrt{\frac{ma_1X}{n}} \in [la_1, X]}\frac{w(la_1)S(nr_1,-m,a_1)}{a_1} 
\l(\frac{a_1}{mn}\r)^{\frac{3}{4}} \l(e\l(\mathcal{H}(X)\r) -e\l(-\mathcal{H}(X)\r)  \r)\\
&+ O\l(X^{\ve}\l(X^{\frac{3}{2}} + \frac{X^2}{H}\r)\r).
\ea
The proof consists of several steps carried out below.
\end{proposition}
\subsection{Localizing the stationary point and decomposition of $B_{1,w}(X,I_4)$} \label{I5two}
For a fixed $a_1$, define
\ba \label{kdef}
k=k(a_1) &:= \max \l\{j\geq 0: \frac{X}{2^j}\geq la_1+1\r\} =\l[\frac{\log (X/(la_1 +1))}{\log 2}\r].
\ea 
By \eqref{kdef}, we must have:  $0 \leq j \leq k \leq J(X) :=\l[\frac{\log (X/2)}{\log 2}\r]$ and $X/2^j \geq la_1 $. We make a dyadic subdivision of the range of $x_0$, i.e., the interval $[1, X^{5/4}]$ in the following manner:
\ba\label{I_j}
& \mathcal{I} _{0} := \l[X+\delta , X^{\frac{5}{4}}\r],
 \mathcal{I} _j := \l[\frac{X}{2^{j}}+\delta, \frac{X}{2^{j-1}}-\delta\r] \text{ for } 1\leq j\leq k, \quad
\mathcal{I}_{k+1}: = [la_1 + \delta, \frac{X}{2^k}-\delta], \quad 
 \mathcal{I}_{k+2}: = [1, la_1 - \delta ].
\ea
For $\delta>0$ is to be chosen later ($\delta=1/100$  suffices), let us define
\be \label{I_j'}
\mathcal{I} '_j := \l(\frac{X}{2^{j-1}}-\delta, \frac{X}{2^{j-1}}+\delta\r) \text{ for } 1 \leq j \leq k+1, 
\mathcal{I} '_{k+2} := (la_1-\delta, la_1+\delta).
\ee 
We write the interval $[1 , X ^{5/4}] $ as a disjoint union of the $\mathcal{I}_j$'s, $\mathcal{I}'_j$s, and decompose the sum as follows: 
\ba \label{B-decomp}
B_{1, w}(X, I_4)&= \frac{1}{2\pi i} \sum_{l|r} \sum_{a_1 >0}\sum_{0 \leq j \leq k+2}\mathop{\sum_{1 \leq m \ll \frac{X^{1+\ve}}{H}}\sum_{1 \leq n \ll \frac{X^{1+\ve}}{lH}}}\frac{1}{m}\frac{w(la_1)S(nr_1,-m,a_1)}{a_1}I_{m,n,a_1}= \mathcal{S}+ \mathcal{T},
\ea
where in the sum $\mathcal{S}$, the stationary point $x_0 \in \cup_j \mathcal{I}_j$ and in the sum $\mathcal{T}$, $x_0 \in \cup_j \mathcal{I}'_j$. Let
\ba \label{F_j}
  \mathcal{F}_j := \l[\frac{X}{2^{j}}, \frac{X}{2^{j-1}}\r] \text{ for } 1\leq j \leq k, \quad  \mathcal{F}_{k+1} := \l[la_1,  \frac{X}{2^k}\r],
\ea
so that \([la_1,X] = \bigcup\limits_{1 \leq j \leq k+1} \mathcal{F}_j\),
where an interval $\mathcal{F}_j$ can intersect its neighbour only at the boundary point. 
Let us define
\be \label{ydef}
y_j := \frac{X}{2^{j}},\text{ for } 0\leq j\leq k+1, \text{ and } y_{k+2} := \frac{la_1}{2} .
\ee
From \eqref{kdef} and \eqref{ydef}, we have, $y _{k+1} \leq la_1\leq y_k $. Note that for $1\leq j\leq k+1$, the interval $\mathcal{I}_j$ is contained in $\mathcal{F}_j$ and the edges 
of  $\mathcal{I}_j$ are away from the edges of $\mathcal{F}_j$. For isolating the stationary phase for $x_0 \in\mathcal{I}_j$ we express the sum $\mathcal{S}$ as 
\ba \label{S-decomp}
\mathcal{S} &= \frac{1}{2\pi i} \sum_{l|r} \mathop{ \sum_{a_1 >0 } \sum_{1 \leq m \ll \frac{X^{1+\ve}}{H}}\sum_{1 \leq n \ll \frac{X^{1+\ve}}{lH}}}_{x_0 \in \mathcal{I} _{0}}\frac{1}{m}\frac{S(nr_1,-m,a_1)w(la_1)}{a_1} \int_{la_1}^{y_0}  e(h(x))\diff x\\
&+ \frac{1}{2\pi i} \sum_{l|r} \mathop{\sum_{a_1 >0 } \sum_{1 \leq m \ll \frac{X^{1+\ve}}{H}}\sum_{1 \leq n \ll \frac{X^{1+\ve}}{lH}}}_{x_0 \in \mathcal{I} _1}\frac{1}{m}\frac{S(nr_1,-m,a_1)w(la_1)}{a_1}\l[\int_{la_1}^{y_1} + \int_{\mathcal{F}_1} \r]e(h(x))\diff x \\
&+\frac{1}{2\pi i}\sum_{l|r} \sum_{a_1 >0}\sum_{1 < j \leq k}\mathop{\sum_{1 \leq m \ll \frac{X^{1+\ve}}{H}}\sum_{1 \leq n \ll \frac{X^{1+\ve}}{lH}}}_{x_0 \in \mathcal{I} _j}\frac{1}{m}\frac{S(nr_1,-m,a_1)w(la_1)}{a_1}\l[\int_{la_1}^{y_j}  + \int _{\mathcal{F}_j}  + \int_{2y_j}^X \r]e(h(x))\diff x\\
&+ \frac{1}{2\pi i} \sum_{l|r} \mathop{ \sum_{a_1 >0} \sum_{1 \leq m \ll \frac{X^{1+\ve}}{H}}\sum_{1 \leq n \ll \frac{X^{1+\ve}}{lH}}}_{x_0 \in \mathcal{I} _{k+1}}\frac{1}{m}\frac{S(nr_1,-m,a_1)w(la_1)}{a_1}\l[\int_{\mathcal{F}_{k+1}} + \int_{2y_{k+1}}^X \r]e(h(x))\diff x\\
&+ \frac{1}{2\pi i} \sum_{l|r} \mathop{ \sum_{a_1 >0 }  \sum_{1 \leq m \ll \frac{X^{1+\ve}}{H}}\sum_{1 \leq n \ll \frac{X^{1+\ve}}{lH}}}_{x_0 \in \mathcal{I} _{k+2}}\frac{1}{m}\frac{S(nr_1,-m,a_1)w(la_1)}{a_1} \int_{2 y_{k+2}}^{X}  e(h(x))\diff x\\
&= \mathcal{S}_1+ (\mathcal{S}_2+ \mathcal{S}_3)+(\mathcal{S}_4+\mathcal{S}_5+\mathcal{S}_6)+(\mathcal{S}_7+ \mathcal{S}_8) + \mathcal{S}_9,
\ea
say. 
\begin{remark}
Note that $la_1>X/2$ iff $k=0$ and in that case, the only sums that occur on the right hand side of the above equation are $\mathcal{S}_1$, $\mathcal{S}_8$ and $\mathcal{S}_9$.
\end{remark}
\subsection{Estimation of $\mathcal{S}_1$, $\mathcal{S}_2$, $\mathcal{S}_4$, $\mathcal{S}_6$, $\mathcal{S}_8$ and $\mathcal{S}_9 $} \label{estiT}
 Let $Q_1(X)  := \mathcal{S}_1 + \mathcal{S}_2 + \mathcal{S}_4$ and 
$ Q_2(X): = \mathcal{S}_6+\mathcal{S}_8+\mathcal{S}_9 $. 
\begin{proposition} \label{S_1U}
We have,
\ba  \nonumber
Q_1(X), Q_2 (X)\ll  X^{\frac{3}{2}+\ve}.
\ea
\end{proposition}
\begin{proof}
We only consider the case of $Q_1(X)$ as $Q_2 (X)$ can be tackled in the same way. 
We have,
\ba \label{S1u}
Q_1(X)=\frac{1}{2\pi i}\sum_{l|r}\ts_{ x_0 \in \mathcal{I}_j} 
 \frac{1}{m}\frac{w(la_1)S(nr_1, -m, a_1)}{a_1} I^+_{m,n,a_1}(y_j),
\ea 
where $I^+_{m,n,a_1}(y_j) := \int\limits_{la_1}^{y_j} e(h(x))\diff x .$
Here the main idea is to subdivide the ranges of $m$, $n$ and $a_1$ according to the position of $x_0$. 
Interchanging the order of summations and by \eqref{ydef}, we have
\ba \label{S1x}
Q_1(X) \ll \sum_{ l|r} \sum_{ 1 \leq j \leq J(X)} \sum_{1 \leq m \ll \frac{X^{1+\ve}}{H}} \sum_{1 \leq n \ll \frac{X^{1+\ve}}{lH}}\sum_{ \substack{ 0 < la_1 \leq y_j \\  \sqrt{\frac{ma_1X}{n}}> y_j +\delta }} \frac{1}{m}\frac{|S(nr_1,-m,a_1)|w(la_1)}{a_1}|I_{m,n,a_1}^{+}(y_j)| .
\ea
The sum over $a_1$ is vacuous if $j$ is so large that $X/2^j <la_1$. We write $
[y_j+\delta, X^{3/2}]= \bigcup_{\substack{Y=2^j\delta\\ 0 \leq j \leq  L(X)}}[y_j+Y, y_j+2Y]$, where $L(X)=O(\log X)$.  We consider two cases. \\
\textbf{Case I:} $\delta \leq Y \leq \frac{y_j}{2}$.
Suppose $y_j + Y \leq x_0 \leq y_j+ 2Y,$ where $\delta \leq Y \leq \frac{y_j}{2}$. This implies
\be \label{n/m}
\frac{n}{m} \leq \frac{a_1X}{(y_j+Y)^2} \ll \frac{1}{l}\l(\frac{X}{y_j}\r)^2,\quad \textnormal{ and }\quad
\frac{n}{mX}(y_j + Y)^2 \leq a_1 \leq  \frac{n}{mX}(y_j +2 Y)^2,
\ee 
and $a_1 \asymp \frac{n{y_j}^2}{mX}$
and the number of such $a_1$ is $O\l(\frac{ny_jY}{mX}+1\r)$. Moreover, from \eqref{n/m}
\ba \label{caseII}
 \frac{n}{a_1}\l(\frac{(y_j+Y)^2}{y_j^2} -1\r) \leq \frac{mX}{y_j^2} - \frac{n}{a_1}\leq \frac{n}{a_1}\l(\frac{(y_j+2Y)^2}{y_j^2} -1\r).
\ea
Hence, by \eqref{k-der} with $k=1$, \eqref{secondderi}, and the first inequality of the last equation, it follows that 
\be \label{lambda1I}
I^+_{m,n,a_1}(y_j) \ll \frac{a_1y_j}{nY}.
\ee
Note that from \eqref{n/m}, we have the condition $nl \ll m(X/y_j)^2$. Therefore, using \eqref{S1x}, the contribution of the terms considered in Case I  is
$O\l( X^{\frac{3}{2}+\ve}\r)$,
by \eqref{Weil}, \eqref{lambda1I} and \eqref{n/m},
where we have extended the definition $y_j=X/2^j$ to $j\geq k+3$ by positivity.\\
\textbf{Case II:} $Y > \frac{y_j}{2}$.
By \eqref{secondderi}, the first inequality of \eqref{caseII} and Prop. \ref{10.1}, this case contributes $O\l(X^{\frac{3}{2}+\ve}\r)$ in $Q_1(X)$ (see \eqref{S1u}). Combining this with the bound in Case I, we prove Prop. \ref{S_1U}.
\end{proof}

\subsection{Estimation of $\mathcal{T}$} We take advantage that the length of each of the intervals is small: $\mathcal{L}(\mathcal{I}_j') = \delta$. 
\begin{proposition}\label{SU}
For $\delta = \frac{1}{100}$, we have
\ba 
\mathcal{T} \ll X^{\ve}\l(\sqrt{X} \l(\frac{X}{H}\r)^{3/2} + X^{3/2}\r). \nonumber
\ea 
\end{proposition}
\begin{proof}
Recalling \eqref{stationarypoint1},  \eqref{I_j'}, \eqref{B-decomp} and  \eqref{ydef}, 
we write $\mathcal{T}$  as 
\ba
&\mathcal{T} = \frac{1}{2\pi i} \sum_{l|r} \sum_{a_1 >0}\sum_{1 \leq j \leq k+2}\mathop{\sum_{1 \leq m \ll \frac{X^{1+\ve}}{H}}\sum_{1 \leq n \ll \frac{X^{1+\ve}}{lH}}}_{ \sqrt{\frac{ma_1X}{n}} \in (2y_j-\delta, 2y_j+ \delta)}\frac{1}{m}\frac{w(la_1)S(nr_1,-m,a_1)}{a_1} \l[\int\limits_{la_1}^{y_j}+ \int\limits_{y_j}^{4y_j} + \int\limits_{4y_j}^{X}\r] e(h(x)) \diff x \\
& = \mathcal{T}_1 + \mathcal{T}_2 + \mathcal{T}_3, \nonumber
\ea
say. By an argument similar to the proof Prop. \ref{S_1U}, it follows  that $
\mathcal{T}_1, \mathcal{T}_3 \ll  X^{3/2+\ve}$. For the estimation of $\mathcal{T}_2$, we bring the $j$-sum outside \eqref{ydef} as before. Thus, we obtain
\ba 
&\mathcal{T}_2 \ll \sum_{l|r}\sum_{1 \leq j \leq J(X) +2} \sum_{1 \leq m \ll \frac{X^{1+\ve}}{H}}\sum_{1 \leq n \ll \frac{X^{1+\ve}}{lH}} \sum_{\substack{ 0 < la_1 \leq 2y_j \\ \sqrt{\frac{ma_1X}{n}} \in (2y_j-\delta, 2y_j+ \delta) } } \frac{w(la_1) }{m}
 \frac{|S(nr_1,-m,a_1)|}{a_1} \l| \int\limits_{y_j}^{4y_j} e(h(x)) \diff x \r|. \nonumber
\ea
Since $x_0 \asymp y_j$, by \eqref{k-der} with $k=2$ and recalling \eqref{secondderi}, we get \(
\int\limits_{y_j}^{4y_j} e(h(x)) \diff x \ll \frac{m^{\frac{1}{4}}X^{\frac{1}{4}}a_1^{\frac{3}{4}}}{\sqrt{2}n^{\frac{3}{4}}}.\)
Now the location of $x_0$ determines the size of $a_1$ (as in \eqref{n/m}) and now, by an application of \eqref{Weil}, we get $
\mathcal{T}_2 \ll\sqrt{X} \l(X/H\r)^{3/2+\ve}.$
\end{proof}
\subsection{Analysis of $\mathcal{S}_3$, $\mathcal{S}_5$, $\mathcal{S}_7$ and the proof of Prop. \ref{G}}\label{station}
Since $x_0 \in \mathcal{I} _j$, and $\mathcal{I} _j \subset \mathcal{F}_j$, we extract the stationary phases by applying Lemma \ref{stationary} and thus  obtain that for $1 \leq j \leq k+1$,
\ba \label{mt}
\int\limits _{\mathcal{F}_j}e(h(x)) \diff x = \frac{m^{\frac{1}{4}}X^{\frac{1}{4}}a_1^{\frac{3}{4}}}{\sqrt{2}n^{\frac{3}{4}}} e\l(-2 \sqrt{\frac{mnX}{a_1}}-\frac{1}{8}\r)+ O(R_1)+O(R_2),
\ea
where
\be  \nonumber
R_1 = \frac{1}{\lambda_2(x_0 -y_j )} + \frac{1}{\lambda_2( 2y_j - x_0 )}, \text{ and }
R_2 = \frac{\mathcal{L}(\mathcal{F}_j)\lambda_4}{\lambda_2^2}+\frac{\mathcal{L}(\mathcal{F}_j)\lambda_3^2}{\lambda_2^3}, 
\ee except in case of $\mathcal{F}_{k+1}$. In this case, the lower end point of the integral is $la_1$, and consequently
\be \nonumber
R_1 =  \frac{1}{\lambda_2(x_0 -la_1 )} + \frac{1}{\lambda_2( 2y_{k+1} - x_0 )} 
\ee
and everything else remains unchanged. As in the proof of Prop. \ref{S_1U}, by introducing dyadic subdivision of the range of the stationary point and analyzing the ensuing restrictions on the variables $m, n$ and $a_1$, we  show that the contributions of the error terms in 
 $\mathcal{S}_3$, $\mathcal{S}_5$ and $\mathcal{S}_7$ is  $O\l(X^{\frac{3}{2}+\ve}\r)$ by \eqref{Weil}.
Furthermore, from \eqref{mt}, we observe that the stationary phase analysis of the integrals in $\mathcal{S}_3,\mathcal{S}_5,\mathcal{S}_7$ give the same main term. Combining all these estimates in one place, we conclude that
\ba \nonumber
&Q_3(X) :=\mathcal{S}_3 + \mathcal{S}_5 + \mathcal{S}_7 = \frac{X^{\frac{1}{4}}}{2\sqrt{2}\pi i}\mathop{\sum_l\sum_m\sum_n\sum_{a_1 >0 } }_{x_0 \in \cup _{1 \leq j \leq k +1} \mathcal{I} _j}\frac{w(la_1)S(nr_1,-m,a_1)}{a_1} \l(\frac{a_1}{mn}\r)^{\frac{3}{4}}e\l(\mathcal{H}(X) \r) + O\l(X^{\frac{3}{2}+\ve}\r).
\ea
Now we extend the range of $x_0$ to the whole interval $[la_1, X ]$ and  subtract the contribution of the extra terms which is $O\l(X^{\frac{3}{2}+\ve}\r)$ by the same argument used in the proof of Prop. \ref{SU}. Thus, we have
\ba \label{S_5}
Q_3(X) &= \frac{X^{\frac{1}{4}}}{2\sqrt{2}\pi i}\mathop{\sum_l\sum_m\sum_n\sum_{a_1 >0}}_{\sqrt{\frac{ma_1X}{n}} \in [la_1, X ]}\frac{w(la_1)S(nr_1,-m,a_1)}{a_1}\l(\frac{a_1}{mn}\r)^{\frac{3}{4}}e\l( \mathcal{H}(X) \r) + O\l(X^{\frac{3}{2}+\ve}\r).
\ea 
Combining Prop. \ref{S_1U}, Prop. \ref{SU}, \eqref{S_5} and also considering the case $m, n<0$, we prove Prop. \ref{G}. 

\subsection{Contribution of $I_5$}
\begin{proposition}\label{BXI_5}
We have,
\ba \nonumber
 B_{1, w}(X, I_5) &=  \int \limits_{X} ^{ X+ H} \frac{w(u)}{u^{1/4}} \mathop{\sum_l\sum_m\sum_n \sum_{a_1 >0}}_{ la_1\leq \sqrt{\frac{ma_1u}{n}} \leq X } \frac{w (la_1)  S(nr_1, -m, a_1)}{\sqrt{2}a_1}\l(\frac{a_1}{mn}\r)^{\frac{1}{4}} \bigg \{e\l(\mathcal{H}(u)\r) + e\l( \mathcal{H}(u)\r)\bigg\} + O\l(  X^{3/2+\ve} \r) .
\ea

\end{proposition}
\begin{proof}
Here we only give a sketch as the proof is very similar to the proof of Prop. \ref{G}. 
From \eqref{hatGintegral} and by change of variable $y=ua_1/x$, we have
\ba 
I_5 &=a_1\int\limits _ X^{X+H} w(u) \int\limits_{la_1}^X \frac{1}{x} w\l(\frac{a_1ul}{x} \r)  e\l(\frac{-nx}{a_1}\r) e\l(\frac{-mu}{x}\r) \diff x \diff u. \nonumber
\ea
Since $u \geq X$, by \eqref{w1} of $w$, we write $I_5$ as 
\ba \nonumber
I_5 &= a_1\int\limits _ X^{X+H} w(u)\l[ \int\limits_{la_1u/X}^X \frac{1}{x}  e\l(\frac{-nx}{a_1}\r) e\l(\frac{-mu}{x}\r) 
+ \int\limits_{la_1}^{la_1u/X} \frac{1}{x} w\l(\frac{a_1ul}{x} \r)  e\l(\frac{-nx}{a_1}\r) e\l(\frac{-mu}{x}\r) \r]\diff x \diff u.
\ea
Let us consider the second integral first. Since $u \leq X+H$, the length of the $x$-integral is bounded by $la_1 H/X$. Hence the second integral is $(a_1 H^2/X)$. For the first integral,  we write
\[
 \int\limits_{la_1u/X}^X=\int\limits_{la_1}^X-\int\limits_{la_1}^{la_1u/X}.
\]
Estimating the second integral trivially, we finally get by Prop. \ref{10.1},
\ba \label{I_5integral}
I_5=a_1\int \limits_{X} ^{ X+ H} w(u) \int\limits _{la_1} ^X  \frac{1}{x}e\l(\frac{-nx}{a_1}\r)  e\l(\frac{-mu}{x}\r) \diff x \diff u
+ O\l(X^{3/2+\ve}\r).
\ea
Let us now compare the above expression for $I_5$ with that of $I_4$. The outside integral over $u$ plays practically no role and we can insert the sum over $m,n,$ and $a_1$ inside and treat $u$ as a constant of size $X$. Thus our phase function is essentially the same as before except that $X$ has been replaced by $u$. However, when we execute the outside integral  at the end, this gives rise to a factor of size $H$. Also, there is a  factor $1/m$ in the case of $I_4$ which is not present here and   there is an extra factor $1/x$ in $I_5$ where $x$ varies from $la_1$ to $X$. In the crucial ranges of the variables; i.e., for $a_1 \asymp X/l$ and $m\asymp X/H$, we see that the two integrals $I_4$ and $I_5$ have the same shape. The same analysis that was carried out in the case of $I_4$  (\S\ref{I5one} to \S\ref{station}) applies here too, with the same kind of subdivisions of the range of the stationary point $x_0 = \sqrt{\frac{ma_1u}{n}}$ and the range of the $x$-integral. However, due to the presence of the factor $1/x$, we need to use Lemma \ref{Huxley} instead of Lemma \ref{stationary}. 
\end{proof}
\subsection{Removal of the sharp-cut condition and the proof of Proposition \ref{ET}}\label{final}
\noindent
\textit{ To  simplify notation, let us, henceforth, denote $a_1$ by $c$ with the understanding that this has no connection the original variable $c$. }\\
At this stage we have expressed $\widehat{G}(n/c, m/c)$ in a form that enabled us to see its oscillating behaviour and its order of magnitude. The condition 
$lc\leq \sqrt{\frac{mcX}{n}} \leq X$ appearing in the sum over $c$ in Prop. \ref{G} is equivalent to saying 
$0 < c \leq \min\{\frac{mX}{nl^2}, \frac{nX}{m}, \frac{X}{l}\}$. However, we can  drop the condition  $a_1\leq \frac{X}{l}$ because of the presence of $w(la_1)$ in our sum and in view of Remark \ref{largea1}. We now need to make the variable smooth. We show how to do that only for the first term in Prop. \ref{G} since the similar analysis holds for the second term in Prop. \ref{G} and the terms in Prop. \ref{BXI_5}. The condition 
$
c \leq \min\l \{\frac{mX}{nl^2}, \frac{nX}{m}\r \} \nonumber
$
can be restated as
\ba \label{othercase}
c \leq  \frac{mX}{nl^2} \textnormal{ when }  m \leq nl ,
\textnormal{ and }
c \leq  \frac{nX}{m} \textnormal{ when } nl \leq m .
\ea 
\begin{proposition}\label{nx/m} We have,
\ba
&\frac{ X^{\frac{1}{4}}e (-1/8)}{2\sqrt{2}\pi i}\mathop{\sum_{l, m,n}} _{nl \leq m}  \sum_ { 0 < c \leq \frac{nX}{m}}  \frac{ S(nr_1, -m, c)}{c}w (lc)
\l(\frac{c}{mn}\r)^{\frac{3}{4}}e\l(-2 \sqrt{\frac{mnX}{c}}\r) \\
&=\frac{ X^{\frac{1}{4}}e (-1/8)}{2\sqrt{2}\pi i} \mathop{\sum_{l, m,n}} _{nl \leq m}  \sum_ { c >0}  \frac{S(nr_1, -m, c)}{c}w \l(\frac{mc}{n} \r)  
\l(\frac{c}{mn}\r)^{\frac{3}{4}}e\l(-2 \sqrt{\frac{mnX}{c}} \r) 
+O\l(X^{1+\ve}\sqrt{H}\r). \nonumber
\ea
An analogous statement holds in the complementary situation, i.e.,
\ba \nonumber
&\frac{ X^{\frac{1}{4}}e (-1/8)}{2\sqrt{2}\pi i} \mathop{\sum_{l, m,n}} _{m \leq nl}  \sum_ { 0 < c \leq \frac{mX}{nl^2}}  \frac{ S(nr_1, -m, c)}{c}w (lc)
\l(\frac{c}{mn}\r)^{\frac{3}{4}}e\l(-2 \sqrt{\frac{mnX}{c}} \r) \\
&=\frac{ X^{\frac{1}{4}}e (-1/8)}{2\sqrt{2}\pi i} \mathop{\sum_{l, m,n}} _{m \leq nl}  \sum_ { c >0}  \frac{ S(nr_1, -m, c)}{c}w \l(\frac{cnl^2}{m}\r) 
\l(\frac{c}{mn}\r)^{\frac{3}{4}}e\l(-2 \sqrt{\frac{mnX}{c}}\r) +O\l(X^{1+\ve}\sqrt{H}\r).
\ea
\end{proposition}
\begin{proof}
In the  case $nl \leq m$, $nX/m \leq X/l$ and hence, by comparing the range of $c$ with the support of $w$ (see \eqref{w0}, \eqref{w1} and \eqref{w-support}), we have (suppressing the sums over $l, m,n$)
\ba 
\sum_ {0 < c \leq \frac{nX}{m}}  w (lc) \cdots
&=\sum_ { 0 < c  \leq \frac{nX}{m}}1. \cdots + O\l(\sum_{\frac{X}{l} \leq c \leq \frac{X+H}{l}}  w (lc) \cdots\r) \\ 
& = \sum_ {c >0} w \l(\frac{mc}{n}\r) \cdots 
+ O\l(  \sum_ {X< \frac{mc}{n} \leq X+H)}w \l(\frac{mc}{n}\r) \cdots\r) + O\l( X^{1+\ve}\sqrt{H}\r)
\nonumber
\ea
and the error term is $O\l( X^{1+\ve}\sqrt{H}\r)$ by \eqref{Weil}. 
\end{proof}

\section{Proof of Proposition \ref{second prop}: initial steps}\label{secondinitial}
It is enough to consider $R_{1,w}(X)$ and $R_{3,w}(X)$ occurring in Prop. \ref{ET} since the other sums  occurring there can be handled in a similar manner.
\subsection{Analysis of $R_{1,w}(X)$} 
We introduce smooth dyadic partitions of unity to the sums over $m, n,$ and $c$ in
 $R_{1,w}(X)$. Let $b_1, b_2,$ and $V$ be smooth functions with support inside the interval $[1,4]$ such that
 \[
 \sum_{j\geq 0}V(x/2^j)=\sum_{j\geq 0}b_i(x/2^j)=1 \ (i=1,2),\quad \text{and} \quad b_i^{(k)}(x), V^{(k)}(x) \ll _k 1 \ (i =1,2)
 \]
for every $x>0$. Thus we have,
\ba \label{4/5}
R_{1,w}(X)=\frac{ X^{\frac{1}{4}} }{2\sqrt{2}\pi i}e\l(-\frac{1}{8}\r)\sum_{l |r}  \mathop{\sum_{M \text{ dyadic}} \sum_{N \text{ dyadic}} \sum_{\substack{ C \text{ dyadic} \\ C \geq \frac{X^{\frac{4}{5}+\ve}}{l} } }} 
\mathop{\sum_m\sum_n}_{nl\leq m} \cdots
 +O\l(X^{\frac{5}{4}+\ve}\l(\frac{X}{H}\r)^{\frac{1}{2}} \r),
\ea
where we have collected the contributions of the smaller $c$'s in the error term after applying \eqref{Weil} (see Prop. \ref{10.1}). 
It is, therefore,  enough to estimate the sum 
\ba  \nonumber
R_1= R_1 (M,N,C,X):= \mathop{\sum_ {m }\sum_{n }}_{nl \leq m} b_1 \l(\frac{m}{M}\r)b_2\l(\frac{n}{N}\r) \sum _ c \frac{S(nr_1,-m,c)}{c}e\l(-2 \sqrt{\frac{mnX}{c}}\r)
w \l(\frac{mc}{n}\r) V\l(\frac{c}{C}\r) g(m,n,c),
\ea
where 
\be \label{gmnc}
g(m,n,c) := \l(\frac{c}{mn}\r)^{\frac{3}{4}}
\ee
and  $l$ is a fixed divisor of  $r$.
Let us define
\be \label{c_0}
C_0 := \max\l\{ 2^{j} :2^{j} \leq   \frac{X^{\frac{4}{5}+\ve}}{l} \r\},
\ee
so that the contribution of the terms corresponding to $C\leq C_0$  is absorbed in the error term after we make the choice $H=\sqrt{X}$ at the end. Recall  that $1 \leq M \ll \frac{X^{1+\ve}}{H}$, $1 \leq N\ll \frac{X^{1+\ve}}{lH}$  and from \eqref{w-support}, we have $C \leq \frac{8NX}{M}$ since $N \leq n \leq 4N$ and $M \leq m \leq 4M$.  With these notations, we have:
\begin{proposition}\label{kuz1}
For fixed $M,N,$ and $C$ such that
\be \label{dya}
C_0 \leq  C \leq \frac{8NX}{M}, 1 \leq M \ll \frac{X^{1+\ve}}{H},\textnormal{ and } 1 \leq N \ll \frac{X^{1+\ve}}{lH},
\ee
we have
\ba \nonumber
&R_1 \ll r_1^{\theta} X^{\frac{3}{4}+\ve} \sqrt{MN} \bigg (\l(\frac{C}{Xl}\r)^{\frac{1}{4}}\frac{1}{\sqrt{N}}+ \l(\frac{C}{Xl}\r)^{\frac{3}{4}}\frac{(M+N)}{N^{\frac{3}{2}}}+ \l(\frac{C}{Xl}\r)^{\frac{5}{4}}\frac{(MN)}{N^{\frac{5}{2}}} \bigg )+  r^{\theta} X^{\frac{3}{4}+\ve}N  \\
& + r_1^{\theta}X^{\ve} \l(\frac{C}{N^2l}\r)^{\frac{3}{4}}\l[ MN \l( \sqrt{\frac{MC}{NH}} + \sqrt{M} + \sqrt{N} + \frac{\sqrt{N}}{(MN)^{\frac{1}{4}}} + \sqrt{\frac{X}{H}} \r)+  (MN)^{\frac{3}{4}} M  \r] + r_1^{\theta} X^{3/4+\ve} M ^{23/32} \\
&\ll r^{\theta} X^{\frac{3}{4}+\ve} X/H.
\ea
\end{proposition}
\subsection{Analysis of $R_{3,w}(X)$  }
The analysis of this sum is similar to that of $R_{1,w}(X)$.
Note that the variable $u$ in the sum $R_{3, w} (X)$ plays a  role analogous to that of 
the parameter $X$ inside the phase function in  $T_{1, w} (X)$ and we can work with a fixed value of $u$ and then integrate trivially. Moreover, the size of $u$ is the same as that of $X$, in fact $u=X+O(H)$ and $H$ is chosen to be $\sqrt{X}$ in the end. 
Therefore, the $u$-integral outside contributes a factor of size $O(X^{-1/4} H)$ which has the same effect as multiplying $X^{1/4}$ to the inner
sum over $l, m,n,$ and $c$. Thus, the essential difference between $R_{1, w} (X)$ and $R_{3, w} (X)$ is that in $R_{3, w} (X)$,
we have the term $(c/mn)^{1/4}$ in the place of $(c/mn)^{3/4}$ in $R_{1, w} (X)$. Note that the condition $nl\leq m$ is now replaced by
$nl\leq mu/X$, but since $u$ and $X$ are of the same size, this change does not make much of a difference. The same comment applies to the argument of the function $w$. 
We  estimate the initial  part of the sum over $c$ by \eqref{Weil} and we truncate the sum at the point
\be\label{c01}
C_0^{\star}:=\max\l\{ 2^{j} :2^{j} \leq   \frac{X^{\frac{2}{3}+\ve}}{l} \r\},
\ee
 for a suitable $\ve>0$. 
By Prop. \ref{10.1}, we have, 
\ba\label{2/3}
\frac{e\l(\frac{1}{8}\r)}{\sqrt{2}} &\int \limits_{X} ^{ X+ H} \frac{w(u)}{u^{1/4}} \mathop{\sum_{l, m,n}}_{ nl\leq \frac{mu}{X}}\sum_{0<c\leq C_0^{\star}} 
 \frac{S(nr_1, -m, c)}{c}
\l(\frac{c}{mn}\r)^{\frac{1}{4}}e\l(-2 \sqrt{\frac{mnu}{c}} \r) w\l(\frac{mc}{n}\frac{u}{X}\r) 
\ll X^{\frac 3 4 +\ve} \l(\frac X H \r)^{\frac 3 2}.
\ea
The choice of $C_0^{\star}$ is dictated by the above bound, since the optimal choice for $H$ is $\sqrt{X}$. 
 For larger values of $c$ we follow the same approach as above and we get the following sum after dyadic subdivision which is analogous to $R_1 (M,N,C,X)$:
\ba \label{g1}
R_3 (M,N,C, u)&:=\mathop{\sum_ {m }\sum_{n }}_{nl \leq  \frac {mu} X} b_1\l(\frac{m}{M}\r) b_2\l(\frac{n}{N}\r) \sum _ c \frac{S(nr_1,-m,c)}{c}e\l(-2 \sqrt{\frac{mnu}{c}}\r)
w \l(\frac{mc}{n}\frac u X\r) V\l(\frac{c}{C}\r) g^{\star}(m,n,c),
\ea
where $g^{\star}(m,n,c):=\l(c/mn\r)^{\frac 1 4}$. We have the proposition below which can be compared with Prop. \ref{kuz1}.
\begin{proposition}\label{kuz2}
Let us fix $u\in [X, X+H]$, where $\sqrt{X}\leq H\leq X$.  For fixed $M,N,$ and $C$ such that
\[
C_0^{\star}\leq  C \leq \frac{8NX}{M}, 1 \leq M \ll \frac{X^{1+\ve}}{H},\textnormal{ and } 1 \leq N \ll \frac{X^{1+\ve}}{lH},
\]
we have
\ba \nonumber
R_3 (M,N,C,X,u) \ll r^{\theta}X^{\ve} \l(X^{\frac{1}{4}}\l(\frac{X}{H}\r)^2 +X^{\frac 1 3}\l(X/H\r)^{\frac 3 2}\r).
\ea
\end{proposition}
We give  the details of the proof of Prop. \ref{kuz1} below and the proof of Prop. \ref{kuz2} 
is almost identical and we only give a very brief sketch of it at the end. 
\section{Application of The Kuznetsov formula} 
\subsection{Decomposition of $R_1 $}
Let us define
\be \label{f-function}
f(t) = f_{m,n,C}(t):= w\l(\frac{4\pi m \sqrt{mnr_1}}{nt}\r)V\l(\frac{4\pi \sqrt{mnr_1}}{tC}\r) g\l(m,n,\frac{4\pi \sqrt{mnr_1}}{t}\r)e\l(- \frac{ X^{\frac{1}{2}} t^{\frac{1}{2}}(mn)^{\frac{1}{4}} }{\sqrt{\pi}r_1^{1/4}}\r),
\ee 
where $g(m,n,c)$ is defined in \eqref{gmnc}. Applying the Kuznetsov formula (Lemma \ref{kuz}) to the  $c$-sum inside $R_1$, we obtain
\ba \nonumber
&R_1 = \Sigma _{\text{disc.}}+ \Sigma _{\text{cont.}} = \mathop{\sum_ {m }\sum_{n }}_{nl \leq m} b_1\l(\frac{m}{M}\r) b_2\l(\frac{n}{N}\r) \sum_{j = 1}^{\infty}\rho_j(nr_1)\rho_j(m)\check{f}_{m,n,C}(\kappa _ j )\\
&+ \mathop{\sum_ {m }\sum_{n }}_{nl \leq m} b_1\l(\frac{m}{M}\r) b_2\l(\frac{n}{N}\r) \frac{1}{\pi}\int\limits_{-\infty}^{\infty} (nmr_1)^{-i\eta}\sigma_{2i\eta}(nr_1)\sigma_{2i\eta}(m)\frac{\cosh(\pi \eta)\check{f}_{m,n,C}(\eta)}{\l|\zeta (1+2i\eta)\r|^2}\diff \eta , \nonumber
\ea
say, where $\check{f}(\eta)$ is defined in \eqref{fcheckkuz}.
From the support of $V$, the argument $t$ of $f(t)$ lies in the interval
\ba \label{lessthan1}
\frac{\pi \sqrt{mnr_1}}{C} < t \leq \frac{4\pi \sqrt{mnr_1}}{C} \leq \frac{16\pi\sqrt{MNr_1}}{C} \ll 1,
\ea
since $ C > C_0 $ (see \eqref{c_0}), $H \geq \sqrt{X}$ (see \eqref{H-size}) and $r=O(X^{1/3 })$ (see Remark \ref{1/3}). We expand $K_{2i\eta}(t)$ 
\[
K_{2i\eta}(t)
=\frac{\pi}{\sinh(2\pi\eta)}
\sum_{\nu=0}^{\infty}
\left[
\frac{\left(\frac{t}{2}\right)^{-2i\eta+2\nu}}{\Gamma(1-2i\eta)\,\nu!\,(1-2i\eta)_\nu}
-
\frac{\left(\frac{t}{2}\right)^{2i\eta+2\nu}}{\Gamma(1+2i\eta)\,\nu!\,(1+2i\eta)_\nu}
\right] = K_{2i\eta, +}(t)- K_{2i\eta, -}(t),
\]
say, where $(a)_n = a(a+1)(a+2)\cdots(a+n-1)$. Accordingly, we write 
\ba \nonumber
\Sigma_{\text{disc.}} = \mathcal{D}_{+}(M,N,C) - \mathcal{D}_{-}(M,N,C),\quad \text{and}\quad \Sigma_{\text{cont.}}= \mathcal{C}_{+}(M,N,C) - \mathcal{C}_{-}(M,N,C),
\ea
where
\ba \label{Dpm}
\mathcal{D}_{\pm} = \mathcal{D}_{\pm}(M,N,C) :=  \mathop{\sum_ {m }\sum_{n }}_{nl \leq m} b_1\l(\frac{m}{M}\r) b_2\l(\frac{n}{N}\r) \sum_{j = 1}^{\infty}\rho _j(nr_1)\rho _j(m) \check{f}_{m,n,C,\pm}(\kappa _j),
\ea
\ba \label{Cpm}
\mathcal{C}_{\pm} = \mathcal{C}_{\pm}(M,N,C) &:= \frac{1}{\pi} \mathop{\sum_ {m }\sum_{n }}_{nl \leq m} b_1\l(\frac{m}{M}\r) b_2\l(\frac{n}{N}\r) \int\limits_{-\infty}^{\infty} (nmr_1)^{-i\eta}\sigma_{2i\eta}(nr_1)\sigma_{2i\eta}(m)\frac{\cosh(\pi \eta)\check{f}_{\pm}(\eta)}{\l|\zeta (1+2i\eta)\r|^2}\diff \eta,
\ea
and $\check{f}_{\pm}(\eta) : = \frac{4}{\pi} \int\limits_ 0 ^{\infty} K_{2i\eta , \pm}(t) f(t) \frac{\diff t}{t}$. We decompose $R_1 $ as $R_1  = \mathcal{D}_{+} - \mathcal{D}_{-}+
\mathcal{C}_{+} - \mathcal{C}_{-}.$ We write 
\ba \label{fcheck}
\check{f}_{\pm}(\eta) & = \frac{4}{\sinh(2\pi \eta)} \sum_{\nu = 0}^{\infty} \frac{2^{-2\nu}}{\Gamma(1\mp 2i\eta)\nu !(1 \mp 2i\eta)\dots (\nu \mp 2i\eta)} I_{\pm}(m,n,C,\nu ,\eta, r_1),
\ea 
where by the change of variables $t = u^2\sqrt{r_1}$,
\be \label{various}
I_{\pm} = I_{\pm}(m,n,C,\nu ,\eta, r_1):= 2 r_1^{\nu \mp i\eta}\int A(u) e(-h_{\pm}(u)) \frac{\diff u}{u},
\ee 
where the weight function
\be \label{au}
A(u)= A_{m,n,\nu,C}(u):= \frac{u^{4 \nu}}{u} w\l(\frac{4\pi m\sqrt{mn}}{nu^2}\r)V\l(\frac{4\pi \sqrt{mn}}{u^2C}\r)g\l(m,n,\frac{4\pi \sqrt{mn}}{u^2}\r),  
\ee 
and the phase function
\be \label{phasefunction}
h_{\pm}(u):=   \frac{X^{\frac{1}{2}} u (mn)^{\frac{1}{4}}}{\sqrt{\pi}} \pm \frac{2 \eta \log u}{\pi} .  
\ee
Under the condition 
$nl \leq m $, we record  the following  bounds that will be required later:
\be \label{wei}
g\l(m,n,\frac{4\pi \sqrt{mn}}{u^2}\r) \ll \l(\frac{C}{N^2l}\r)^{\frac{3}{4}},
\text{ and }
\|A \|_{\infty} \ll \l(\frac{\sqrt{MN}}{C}\r)^{2\nu-1/2} \l(\frac{C}{N^2l}\r)^{\frac{3}{4}} . 
\ee

\subsection{Estimation of $\mathcal{D}_{+} $ and $\mathcal{C}_{+}$}

Here, the phase function $-h_{+}$ has no stationary point and so \eqref{k-der} and the Weyl law suffices. 
\begin{proposition}\label{fun}
We have, for any $\ve >0$
\ba \nonumber
 \check{f}_{+}(\eta) \ll 
\begin{cases} 
 \l(\frac{C}{N^2l}\r)^{\frac{3}{4}}\frac{1}{\eta ^{\frac{3}{2}}}e^{-\pi |\eta|}  & \text{if }  \eta > 0, \\
 X^{-100}, & \text{if } \eta \gg \frac{MC X^{\ve}}{NH} .
\end{cases}
\ea
\end{proposition}
\begin{proof} We have $ h'_{+}(u)  \geq 2 \eta /\pi u \gg \eta/\l(\sqrt{MN}/C\r)^{1/2}$.
For $\eta  \gg \frac{MC X^{\ve}}{NH}$, by Lemma \ref{BKY1} that
 the integral $I_{+}$ is negligibly small. For $\eta >0$. We apply \eqref{k-der} with $k=1$ and the Stirling asymptotics for the Gamma function. The proposition follows.
\end{proof}
\begin{lemma}\label{Rankin-Selberg}
For any integer $q\geq 1$, the following bound holds. 
\be\label{RS}
\sum_{ 1 \leq n \leq N }\frac{|\rho_j(nq)|^2}{\cosh(\pi \kappa_j)}  \ll q^{\theta+\ve} N(N\kappa_j)^{\ve}.
\ee
\end{lemma}
\begin{proof}
By \eqref{heckerelation}, Remark  \ref{kim} and Lemma \ref{hec}, we have
\ba \label{N}
\sum_{1 \leq n \leq N }\frac{|\rho_j(nq)|^2}{\cosh(\pi \kappa_j)} 
  \ll \kappa_j^{\ve} \sum_{ 1 \leq n \leq N }\l|\sum_{k|(n,q)}\mu(k)\lambda _j\l(\frac{q}{k}\r)\lambda _j\l(\frac{n}{k}\r) \r| 
\ll q^{\theta+\ve} N(N\kappa_j)^{\ve} \nonumber
\ea
\end{proof}
\begin{proposition}\label{discrete}
Let  $1 \ll K_1,K_2 \ll X^2$, and $\lambda \geq 0$. For any fixed $c_1, c_2>0$ and any positive integer $q$, we have,
\ba \nonumber
\Theta_{c_1,c_2}(K_1,K_2,\lambda)
&:=\sum_{\substack{ c_1K_1 \leq \kappa_j \leq c_2K_2  }}  \mathop{\sum_ {m \sim M} \sum_ {n \sim N}} | \rho_j(nq)\rho_j(m)|\frac{e^{-\pi |\kappa_j|}}{\kappa _j ^{\lambda}} \ll
\begin{cases}
 q^{\theta} MN K_2^{2-\lambda} & \text{if}\quad \lambda \leq 2\\
q^{\theta} MN & \text{if}\quad \lambda > 2 .
\end{cases}
\ea
\end{proposition}

\begin{proof}
 We subdivide the $j$-sum into dyadic pieces with $\kappa_j \asymp K$ such that $c_1K_1 \leq K \leq c_2K_2$, and apply the Cauchy-Schwarz inequality to both the sums over $m$ and $n$. After that we apply \eqref{RS} to sums over $m$ and $n$ and later the Weyl law to the $\kappa_j$-sum.
\end{proof}
\begin{proposition}\label{prop4} For $1 < \lambda \leq 2 $, we have
	\ba \nonumber
	\Theta(\lambda) &:= \mathop{\sum_ m \sum_ n}_{nl \leq m} b_1 \l(\frac{m}{M}\r)  b_2 \l(\frac{n}{N}\r) \int\limits _{0}^{ \infty}(nr_1m)^{-i\eta}\sigma_{2i\eta}(nr_1)\sigma_{2i\eta}(m)\frac{e^{-\pi|\eta|}\cosh(\pi \eta)}{\eta^{\lambda}\l|\zeta (1 + 2i\eta)\r|^2}\diff \eta \ll  MNX^{\ve}. 
	\ea
\end{proposition}
\begin{proof}
	We break the integral into two parts $\l[\int_0^{\ve}+\int_{\ve}^ {\infty}\r]$ and the proposition follows from \cite[(2.1.16)]{T} that for $0 < \eta <\ve $ and from \cite[(3.6.5)]{T}, for $\eta \geq \ve$.
\end{proof}
\begin{proposition} \label{D+}
We have,
\ba \nonumber
\mathcal{D}_{+} \ll r_1^{\theta} \l(\frac{C}{N^2l}\r)^{\frac{3}{4}}MN  \l( \frac{MCX^{\ve}}{NH}\r)^{\frac{1}{2}}, \textnormal{ and } \mathcal{C}_{+}  \ll \l(\frac{C}{N^2l}\r)^{\frac{3}{4}} MNX^{\ve}.
\ea
\end{proposition}
\begin{proof}
Using \eqref{Dpm}, Prop. \ref{fun},  Prop. \ref{discrete}, and \eqref{dya}, we easily get the bound for $\mathcal{D}_{+}$. The bound for $\mathcal{C}_+$ is clear from Prop. \ref{fun}, \eqref{Cpm} and Prop. \ref{prop4}. 
\end{proof}

\section{Estimation of $\mathcal{D}_{-}$ for $C_0\leq C\leq C_1$}
The sum $\mathcal{D}_{-}$ is considerably more difficult to estimate than $\mathcal{D}_{+}$ because the phase function $-h_{-} (u)$ (see \eqref{phasefunction}) can have a stationary point inside the range of the integral $I_{-}$ which is given by $u_0$. We have, 
\ba \label{h''u}
h'_-(u) = \frac{X^{\frac{1}{2}} (mn)^{\frac{1}{4}}}{\sqrt{\pi}} - \frac{2 \eta}{u \pi} , \,
h''_-(u) = \frac{2\eta}{\pi u^2},\, \text{and} \, u_0=u_0(m,n,X,\eta) := \frac{2\eta}{\sqrt{\pi} X^{\frac{1}{2}}  (mn)^{\frac{1}{4}} }.
\ea 
\subsection{Further truncation of the $c$-sum}\label{c1section}
We shall consider two cases, depending on whether or not 
the factor $w(\frac{4\pi m \sqrt{mn}}{n u^2})$ of $A(u)$ (see \eqref{au}) takes the value $1$. 
Let us define
\be \label{c_1}
C_1 := \max \l\{2^j : 2^j \leq \frac{XN}{100M} \r\}. 
\ee
From \eqref{w1} and the support of $V$, we analyze the following integral for $C_0 < C \leq C_1$.
\[
I_{-}= 2 r_1^{\nu + i\eta}\int A(u) e(- h_{-}(u)) \diff u, \textnormal{ where }
A(u)=  \frac{u^{4 \nu }}{u} V\l(\frac{4\pi \sqrt{mn}}{u^2 C}\r)g\l(m,n,\frac{4\pi \sqrt{mn}}{u^2}\r).
 \]
\begin{proposition}\label{main-d-minus}
For $C_0 \leq C \leq C_1$, we have the bound
\ba \nonumber
\mathcal{D}_{-} & \ll  r^{\theta} X^{\frac{3}{4}+\ve} \sqrt{MN} \bigg (\l(\frac{C}{Xl}\r)^{\frac{1}{4}}\frac{1}{\sqrt{N}}+ \l(\frac{C}{Xl}\r)^{\frac{3}{4}}\frac{(M+N)}{N^{\frac{3}{2}}}+ \l(\frac{C}{Xl}\r)^{\frac{5}{4}}\frac{(MN)}{N^{\frac{5}{2}}} +N \bigg ).
\ea
\end{proposition}
Unlike the previous
case of $\mathcal{D}_{+}$, an application of the Weyl law is not sufficient here. We need to obtain an explicit  asymptotic expansion of the integral $I_{-}$ by applying Lemma \ref{BKY2}.  This is essential because if we estimate $I_{-}$ then our final bound will not be satisfactory. We will see that it is enough to consider the leading term $\mathcal{L}$ of this asymptotic expansion (see \eqref{lt}.
The  phase $e(-h(u_0))$ occurring in $\mathcal{L}$ simplifies to the term $(mn)^{-i\kappa_j}$ up to multiplication by a 
 non-oscillatory function that does not involve $m$ and $n$. We  take advantage of  the cancellations in the sums over $m$ and $n$ arising out of this term. However, we need to make the variables $m$ and $n$  free and separate first. This is achieved by applying the Perron formula to remove the condition $nl\leq m$  and then by Mellin inversion to free them from the function $V$. Finally, once the two variables are separated, we apply the Cauchy-Schwarz inequality and the the Large Sieve  for  twisted 
Fourier coefficients $\rho_j (n)n^{-i\kappa_j}$ (\ref{ls}). 
\subsection{Beginning of the proof of Proposition \ref{main-d-minus} }
By Lemma \ref{BKY2}, we have the expansion
\be \label{ex}
I_{-} = \frac{2r_1^{\nu  + i\eta}}{\sqrt{h''_-(u_0)}}e\l(-h_-(u_0) - \frac{1}{8}\r) \sum\limits_{k \leq 3\delta^{-1}A} p_k(u_0)+ O_{A,\delta}\l({\eta}^{-A}\r),\ee where
\be \label{p0u0}
p_k(u_0) = \frac{1}{k!}
\l(\frac{i}{4\pi h''_-(u_0)}\r)^k G^{(2k)}(u_0), \, G(u) = A(u)e(-H(u)),\, H(u)= h_-(u)- h_-(u_0)-\frac{1}{2} h''_-(u_0)(u-u_0)^2. 
\ee 
where $A$ and $\delta$ are arbitrary subject to the conditions $A>0$ and $0<\delta<1/10$. We define
\be \label{pzeroA}
G(u) = A(u)e(-H(u)),\quad \text{where}\quad H(u)= h_-(u)- h_-(u_0)-\frac{1}{2} h''_-(u_0)(u-u_0)^2. 
\ee
Now we write $\mathcal{D}_-$ explicitly using \eqref{Dpm}, \eqref{fcheck} and \eqref{ex} $
\mathcal{D}_-= \sum\limits_{k \leq 3\delta^{-1}A} \mathcal{D}_{-,k} + O\l(\mathcal{D}_{-, Z}\r),$
where
\ba \label{D_k}
\mathcal{D}_{-,k}& := 4\mathop{\sum_ {m }\sum_{n }}_{nl \leq m} b_1\l(\frac{m}{M}\r) b_2\l(\frac{n}{N}\r) \sum_{j = 1}^{\infty}\rho _j(nr_1)\rho _j(m)\frac{1}{\sinh(2\pi \kappa_j)\Gamma(1+ 2i\kappa_j)} \\
&\times\sum_{\nu = 0}^{\infty} \frac{2^{-2\nu}}{\nu !(1 + 2i\kappa_j)\dots (\nu + 2i\kappa_j)}\l[\frac{2r_1^{\nu  + i\eta}}{\sqrt{h''_-(u_0)}}e\l(-h_-(u_0) - \frac{1}{8}\r) p_k(u_0)\r],
\ea
and  
\ba \label{D_Z}
\mathcal{D}_{-,Z} &:= 4\mathop{\sum_ {m }\sum_{n }}_{nl \leq m} b_1\l(\frac{m}{M}\r) b_2\l(\frac{n}{N}\r) \sum_{j = 1}^{\infty}\l|\frac{\rho _j(nr_1)\rho _j(m)\kappa_j ^{-A}}{\sinh(2\pi \kappa_j)\Gamma(1+ 2i\kappa_j)}\r|
&\sum_{\nu = 0}^{\infty} \frac{2^{-2\nu}r_1^{\nu   }}{\nu !\l|(1 + 2i\kappa_j)\dots (\nu + 2i\kappa_j)\r|}.
\ea
We first consider the case $k=0$ in \eqref{D_k}. By \eqref{h''u} and \eqref{gmnc}, let us first define
\ba \label{lt}
\mathcal{L} = \mathcal{L}(m,n,\kappa _ j, \nu):=   \frac{\sqrt{2}\pi^2X^{3/4}}{\kappa_j ^2} e\l( -\frac{2\kappa_j}{\sqrt{\pi}} - \frac{1}{8}\r)\l( \frac{2\kappa_j}{\sqrt{\pi}X^{1/2}}\r)^{4i\kappa_j}(mn)^{-i\kappa_j} \l( \frac{2\kappa_j}{\sqrt{\pi}X^{1/2}(mn)^{1/4}}\r)^{4\nu}
V\l( \frac{\pi^2 Xmn}{\kappa_j ^2C}\r)  , 
\ea
 From the support of $V$, the value of $u_0$, the condition $nl \leq m $, the effective range of $\kappa_j$ is given by
 \be\label{eta2}
\frac{\pi}{2}\l(\frac{mnX}{C}\r)^{\frac{1}{2}} \leq \kappa_j \leq \pi\l(\frac{mnX}{C}\r)^{\frac{1}{2}}; \quad \textnormal{in particular,} \quad 
\frac{\pi}{2} N \l(\frac{Xl}{C}\r)^{\frac{1}{2}} \leq \kappa_j \leq 4 \pi \l(\frac{MNX}{C}\r)^{\frac{1}{2}} .
\ee
Hence, we write $\mathcal{D}_{-, 0}$ in the following way.
\ba \label{D_0-expression}
\mathcal{D}_{-, 0} := 4 \mathop{\sum_ m \sum_ n}_{nl \leq m}b_1 \l(\frac{m}{M}\r)  b_2 \l(\frac{n}{N}\r) \sum_{\nu = 0}^{\infty}  
\sum_{\frac{\pi}{2} N \l(\frac{Xl}{C}\r)^{\frac{1}{2}} \leq \kappa_j \leq 4 \pi \l(\frac{MNX}{C}\r)^{\frac{1}{2}} }  \rho_j(nr_1)\rho_j(m) \mathcal{V}(\kappa _j, \nu,r_1) \mathcal{L},
\ea
where
\ba  \label{nu-bound-original}
\mathcal{V} = \mathcal{V}(\kappa _j, \nu,r_1) &:= \frac{1}{\sinh(2\pi \kappa _ j)} \frac{2^{-2i\kappa_j -2 \nu}r_1^{\nu  + i\kappa_j }}{\Gamma(1 + 2i\kappa _ j) \nu !(1+2i\kappa _ j)\dots (\nu + 2i\kappa _ j)}  
\ll \frac{e^{-\pi \kappa_j} r_1^{\nu} }{ \nu !\sqrt{\kappa_j}},
\ea
which follows from the Stirling asymptotic (see \cite[Eq. (5.113)]{IK}). Let us define 
\ba\label{Ttilde}
\tilde{\mathcal{T} } = \tilde{\mathcal{T} }(m,n , C):=\sum_{\nu = 0}^{\infty}  
\sum_{\frac{\pi}{2} N \l(\frac{Xl}{C}\r)^{\frac{1}{2}} \leq \kappa_j \leq 4 \pi \l(\frac{MNX}{C}\r)^{\frac{1}{2}} }| \rho_j(nr_1)\rho_j(m) \mathcal{V} \mathcal{L}|.
\ea

\subsection{Removal of the condition  $nl\leq m$}
We capture $ nl \leq m$ by  Lemma \ref{perron} and write
\[
\mathbf{1}_{nl\leq m}=\frac{1}{2\pi i} \int\limits _{\sigma_1- i T}^{\sigma_1 + iT} 
\l(\frac{m}{nl}\r)^{s_1}\frac{\diff s_1}{s_1}  -\frac{1}{2}\mathbf{1}_{nl=m}+
O\l(\l(\frac{m}{nl}\r)^{\sigma_1}\min\l(1, \frac{1}{T|\log(m/nl)|}\r)+\frac{1}{T}\r),
\]
where
$\sigma_1=\Re s_1=\ve$  and $T$ is a  parameter that will be chosen later. 
Note that when $m=nl$, the corresponding sum will be weighted by the factor $1/2$ instead of $1$. Thus we get, 
\ba \label{removal}
\mathcal{D}_{-, 0}&= \frac{4}{2\pi i}  \mathop{\sum_ m \sum_ n}\cdots \int\limits _{\sigma_1- i T}^{\sigma_1 + iT} 
\sum_{\nu = 0}^{\infty}  
\sum_{\frac{\pi}{2} N \l(\frac{Xl}{C}\r)^{\frac{1}{2}} \leq \kappa_j \leq 4 \pi \l(\frac{MNX}{C}\r)^{\frac{1}{2}} }  \rho_j(nr_1)\rho_j(m) \mathcal{V} \mathcal{L} \l(\frac{m}
{nl}\r)^{s_1} \frac{\diff s_1}{s_1}+O(E_1+E_2+E_3) ,
\ea
where
\ba \nonumber
E_1 :=\frac{1}{T} \mathop{\sum_ {m \sim M } \sum_{n \sim N}}_{m \neq nl}
\l(\frac{m}{nl}\r)^{\sigma_1}\frac{\tilde{\mathcal{T} } }{|\log (m/nl)|}, \,
E_2 :=\mathop{\sum_{m \sim M} \sum_{n \sim N}}_{m=nl} \int\limits _{ \sigma_1-iT}^{\sigma_1+iT} 
\tilde{\mathcal{T} }\frac{\diff |s_1|}{|s_1|}  ,\, E_3 :=\frac{1}{T} \mathop{\sum_ {m \sim M} \sum_ {n \sim N}}_{m=nl} \tilde{\mathcal{T} }, 
\ea
\begin{proposition}\label{removalprop}
We have,
\ba
E_1 \ll r _1^{\theta}X^{3/4}\frac{X^{1+\ve}}{TH} MN, \quad
E_2 \ll {r_1}^{\theta} X^{3/4+\ve} \sqrt{MN}, \quad 
E_3 \ll r_1^{\theta} \frac{X^{3/4+\ve}}{T} \sqrt{MN}. \nonumber 
\ea
\end{proposition}
\begin{proof}
First we consider $E_1$. Since $m\neq nl$, we have the bound $ |\log(m/nl)|\gg 1/m$ by the Taylor series expansion
of $\log (1+x)$ for $|x|<1$. Hence 
\ba \nonumber
E_1 &\ll \frac{X^{1+\ve}}{TH}\mathop{\sum_ m \sum_ n}b_1 \l(\frac{m}{M}\r)  b_2 \l(\frac{n}{N}\r)
\tilde{T}.
\ea
From the support of $V$, $u_0\asymp \sqrt{MN}/C$. Hence from \eqref{lt} and \eqref{lessthan1}, we have
\ba \label{Ttilde_estimation}
\tilde{\mathcal{T} }  &\ll X^{3/4} \sum_{\frac{\pi}{2} N \l(\frac{Xl}{C}\r)^{\frac{1}{2}} \leq \kappa_j \leq 4 \pi \l(\frac{MNX}{C}\r)^{\frac{1}{2}} }| \rho_j(nr_1)\rho_j(m)| \frac{e^{-\pi \kappa_j}}{\kappa_j^{5/2}},
\ea
This yields the bound for $E_1$ by Prop. \ref{discrete}. For $E_2$, \eqref{Ttilde_estimation} and trivial estimation of the integral give us
\ba \nonumber
E_2 &\ll X^{3/4+\ve} \sum_ n b_1 \l(\frac{nl}{M}\r)  b_2 \l(\frac{n}{N}\r)\sum_{\frac{\pi}{2} N \l(\frac{Xl}{C}\r)^{\frac{1}{2}} \leq \kappa_j \leq 4 \pi \l(\frac{MNX}{C}\r)^{\frac{1}{2}} } |\rho_j(nl)\rho_j(nr_1)|\frac{e^{-\pi \kappa_j}}{\kappa_j^{5/2}}.
\ea
Now, by the Cauchy-Schwarz inequality applied to the $n$-sum followed by an application of Lemma \ref{Rankin-Selberg} and Lemma \ref{weyl}, we estimate $E_2$. The sum $E_3$ is estimated in the same manner.
\end{proof}

\subsection{Analysis of the main term in \eqref{removal}}\label{leading-term}
Now we introduce a dyadic subdivision of the sum over $\kappa_j$ and 
separate the variables $m$ and $n$ in the argument of the function $V$ by Mellin inversion.
Thus, using \eqref{removal}, \eqref{lt}, \eqref{au},
\eqref{gmnc}, \eqref{h''u}, we get
\ba \label{cs}
&\mathcal{D}_{-, 0}
\ll  X^{\frac{3}{4}}\int\limits _{\sigma-iT}^{\sigma +iT} \int\limits _{(\sigma_2)}\l(\frac{X}{C}\r)^{-\sigma_2}
|\tilde{V}(s_2)|U(M,N, C)\frac{\diff |s_1|}{|s_1|}\diff |s_2| +E_1+E_2+E_3,
\ea
where 
\ba\label{U}
&U(M,N, C):=\sum_{\substack{ K \text{ dyadic} \\ \frac{\pi}{2} N \l(\frac{Xl}{C}\r)^{\frac{1}{2}}  \leq K \leq 4 \pi \l(\frac{MNX}{C}\r)^{\frac{1}{2}} }} \sum_{\kappa_j \sim K}\frac{ \kappa_j^{2\sigma_2}}{\kappa_j^{2}} \Bigg|\sum_{\nu = 0}^{\infty} \mathcal{V} \mathop{\sum_ m \sum_ n} \dots \l(\frac{2\kappa_j}{\sqrt{\pi}X^{\frac{1}{2}}(mn)^{\frac{1}{4}}}\r)^{4\nu} \l(\frac{m}{nl}\r)^{s_1}\frac{\rho_j(nr_1) \rho_j(m)}{(mn)^{s_2+i\kappa_j}}
\Bigg|\\
&\ll \sum_{\substack{ K \text{ dyadic} \\ \frac{\pi}{2} N \l(\frac{Xl}{C}\r)^{\frac{1}{2}}  \leq K \leq 4 \pi \l(\frac{MNX}{C}\r)^{\frac{1}{2}} }} \frac{ K^{2\sigma_2}}{K^{\frac{5}{2}} e^{\pi \kappa_j}  } \sum_{\nu = 0}^{\infty} \l(\frac{2 K}{\sqrt{\pi} X^{\frac{1}{2}}}\r)^{4\nu}\frac{r_1^{\nu}}{ \nu !}  \sum_{\kappa_j \sim K} 
\l|\sum_ {m \sim M} \sum_ {n \sim N}\l(\frac{m}{nl}\r)^{s_1} \frac{\rho_j(nr_1) \rho_j(m)}{(mn)^{s_2++ \nu+i\kappa_j}}\r| 
\ea
by  \eqref{nu-bound-original} and integration by parts. By the Cauchy–Schwarz inequality applied to the $\kappa_j$-sum, followed by an application of the large sieve inequality \eqref{ls},
and taking  $ \sigma_2 = \ve$, we obtain
\ba \nonumber
 U(M,N,C) \ll  r_1^{\theta}X^{\ve}\sqrt{MN} \sum_{\substack{ K \text{ dyadic} \\ \frac{\pi}{2} N \l(\frac{Xl}{C}\r)^{\frac{1}{2}}  \leq K \leq 4 \pi \l(\frac{MNX}{C}\r)^{\frac{1}{2}} }} \frac{  (K^2+M^2)^{\frac{1}{2}}(K^2+N^2)^{\frac{1}{2}}}{K^{\frac{5}{2}}}\sum_{\nu = 0}^{\infty} \l( \frac{16K^4 r_1}{ \pi ^2 X^{2}MN}\r)^{\nu} .
\ea
Here we have used Hecke multiplicativity (see \eqref{heckerelation}) to separate the factor $r_1$ in $\rho_j (n r_1)$ in the manner of the 
proof of Lemma \ref{Rankin-Selberg}. 
Note that since $\sqrt{MNr_1}/C\ll 1$ (see \eqref{lessthan1}), the sum over $\nu$ is absolutely bounded. Now we execute the integrals in \eqref{cs} over $s_1$ and $ s_2$. Because of the decay of $\tilde{V}$, the integral over $s_2$ presents no difficulty and the integral over $s_1$ contributes a factor which is $O(\log X)$  if we choose $T=X^{10}$ (see \eqref{removal}). Hence, by Prop. \ref{removalprop} with this choice of $T$, we have the bound 
\ba \label{dis1}
\mathcal{D}_{-, 0} & \ll r_1^{\theta} X^{\frac{3}{4}+\ve} \sqrt{MN} \bigg (\l(\frac{C}{Xl}\r)^{\frac{1}{4}}\frac{1}{\sqrt{N}}+ \l(\frac{C}{Xl}\r)^{\frac{3}{4}}\frac{(M+N)}{N^{\frac{3}{2}}}+ \l(\frac{C}{Xl}\r)^{\frac{5}{4}}\frac{(MN)}{N^{\frac{5}{2}}} \bigg )
 +{ r_1}^{\theta} X^{\frac{3}{4}+\ve}N,
\ea
since $K \geq \frac{\pi}{2} N \l(\frac{Xl}{C}\r)^{\frac{1}{2}}$.
\begin{remark} \label{8/5}
If we did not use the large sieve (Lemma \ref{ls}) here, we would have obtain the bound  
$
U(M,N, C) \ll (MN)^{3/4},
$
and putting this in \eqref{cs},  Prop. \ref{removalprop} shows that the error term can be no better than
$O(X^{8/5+\ve})$.
The large sieve  improves the bound  to
$O((MN)^{-1/4}(M+N))$  and this gives us a saving which is of size $(X/H)^{3/4}\asymp X^{3/8}$ for our final choice $H=\sqrt{X}$.
\end{remark}
\subsection{Estimation of  $\mathcal{D}_{-,k}$ \textbf{for} $k>0$}\label{estimate4}
A routine calculation shows that for $k=1,2$, we have
\ba \label{pkbound}
||p_k||_{\infty}\ll \frac{1}{\eta}||p_0||_{\infty},\quad \text{and}\quad ||p_0||_{\infty} \ll ||A||_{\infty} \ll  \frac{\l(\frac{\sqrt{MN}}{C}\r)^{2\nu}}{\l(\frac{\sqrt{MN}}{C}\r) ^{1/2}} \l(\frac{C}{N^2l}\r)^{\frac{3}{4}},
\ea  
where we use \eqref{wei}, \eqref{p0u0} and \eqref{pzeroA}. Now, \eqref{pkbound}, \eqref{ex} and \eqref{derivative} with $j = 0$ imply that the terms in the asymptotic expansion of $I_{-}$ (see \eqref{ex})  corresponding to $k\geq 1$ are bounded by $\l(\frac{\sqrt{MNr_1}}{C}\r)^{2\nu} \l(\frac{C}{N^2l}\r)^{\frac{3}{4}} \frac{1}{\eta^{3/2}}$. Hence, by \eqref{D_k}, \eqref{lessthan1} and Prop. \ref{discrete}, we have $\mathcal{D}_{-,k}
\ll \l(\frac{C}{N^2l}\r)^{\frac{3}{4}} \l(r_1^{\theta+ \ve}MN \r)$ for any $k\geq 1$.
\subsection{Estimation of $\mathcal{D}_{-,Z}$}
As before, taking $A = 2$, and using \eqref{dya} \eqref{lessthan1}, \eqref{discrete}, \eqref{D_Z}, we have $\mathcal{D}_{-,Z}\ll r^{\theta}\l(\frac{X}{H}\r)^{2+\ve}$.
Bringing together \eqref{dis1}, all the estimates above and taking $\delta = 1/20$, we complete the proof of Prop. \ref{main-d-minus}.

\section{Estimation of $\mathcal{D}_{-}$ for $ C > C_1$}
In this case, we cannot apply Lemma \ref{BKY2}  since the condition \eqref{bound} need not hold and instead, we apply Lemma \ref{Huxley} which is suitable for the weighted stationary phase analysis of integrals over finite intervals. 
\begin{proposition} \label{c>c1}
For $C_1 < C \leq \frac{8NX}{M}$, we have
\ba \nonumber
&\mathcal{D}_{-} \ll r_1^{\theta} X^{\frac{3}{4}+\ve} \sqrt{MN} \bigg (\l(\frac{C}{Xl}\r)^{\frac{1}{4}}\frac{1}{\sqrt{N}}+ \l(\frac{C}{Xl}\r)^{\frac{3}{4}}\frac{(M+N)}{N^{\frac{3}{2}}}+ \l(\frac{C}{Xl}\r)^{\frac{5}{4}}\frac{(MN)}{N^{\frac{5}{2}}} \bigg ) +  r^{\theta} X^{\frac{3}{4}+\ve}N\\
&+ r_1^{\theta} X^{3/4+\ve} M ^{23/32} + r_1^{\theta}X^{\ve} \l(\frac{C}{N^2l}\r)^{\frac{3}{4}} \l( MN \l( \sqrt{M} + \sqrt{N} + \frac{\sqrt{N}}{(MN)^{\frac{1}{4}}} + \sqrt{\frac{X}{H}}\r) +(MN)^{\frac{3}{4}} M \r) .
\ea
\end{proposition}
From the support of $w$, The integral we need to analyze now is
\ba \label{I-}
I_{-} = 2 r_1^{\nu + i\eta} \int\limits_{\alpha}^{\beta} A(u) e(- h_{-}(u)) \diff u,\, \textnormal{where } \alpha := \frac{\sqrt{4\pi}\sqrt{m}(mn)^{\frac{1}{4}}}{\sqrt{(X+H)n}} \text{ and }  \beta := \frac{\sqrt{4\pi}(mn)^{\frac{1}{4}}}{\sqrt{C_1}}.
\ea 
First we will apply Lemma \ref{BKY1} to our integral to observe that when $\eta $ is sufficiently large the integral is negligibly small. 
\begin{proposition}\label{ra1}
For any $\ve>0$, if $\eta\gg  X^{1+\ve}/{H}$, then we have the bound
\be 
I_{-} \ll_{\ve} X^{-100}. 
\ee
\end{proposition} 
\begin{proof}
Without loss of generality, we may assume $M\ll X^{1+\frac{\ve}{2}}/H$.  Since $C_1\asymp NX/M$ (see \eqref{c_1}) 
\be\label{beta-M-relation}
\beta \sqrt{X} (MN)^{1/4} \asymp M. 
\ee 
Now, by \eqref{h''u}, $ \eta > \sqrt{\pi}\beta X^{\frac{1}{2}} (mn)^{\frac{1}{4}} \iff u_0 > 2\beta$. By our assumption that $\eta \gg \frac{X^{1+\ve}}{H}$, we see that $\eta\gg MX^{\frac{\ve}{2}}$ and hence by \eqref{beta-M-relation},
$\eta > \sqrt{\pi}\beta X^{\frac{1}{2}} (mn)^{\frac{1}{4}} $. In this case, it is easy to verify that
\be  \nonumber
|h'_-(u)| > |h'_-(\beta)| = \frac{2 \eta}{\beta \pi}-\frac{X^{\frac{1}{2}} (mn)^{\frac{1}{4}}}{\sqrt{\pi}} > \frac{\eta}{\beta \pi},
\ee
and hence, by Lemma \ref{BKY1}, the result follows. 
\end{proof}
\subsection{Decomposition of $\mathcal{D}_{-}$}
Here we may assume that $\eta \ll \frac{X^{1+\ve}}{H}$. We divide the integral as follows:
\ba \label{I0refer}
I_{-} &= 2r_1^{\nu + i\eta}\l[\int_{\frac{2\sqrt{\pi}\sqrt{m}(mn)^{1/4}}{\sqrt{n(X+H)}}}^{\frac{2\sqrt{\pi}\sqrt{m}(mn)^{1/4}}{\sqrt{Xn}}} +  \int_{\frac{\sqrt{4\pi}\sqrt{m}(mn)^{1/4}}{\sqrt{Xn}}}^{\frac{\sqrt{4\pi}(mn)^{1/4}}{\sqrt{C_1}}} \r] A(u)  e(h_{-}(u)) \diff u = \sum\limits_{\mu = 0,1}I_{-,\mu}
 \ea
say. In the second integral, the function $w$ which occurs inside $A(u)$ (see \eqref{au})
is identically equal to $1$. Correspondingly, we write 
\ba \label{calD-}
\check{f}_{-}(\eta) = \sum\limits_{\mu= 0,1}\check{f}_{-,\mu}(\eta),\quad \text{and} \quad \mathcal{D}_{-} = \sum\limits_{\mu = 0,1}\mathfrak{D}_{-,\mu}.
\ea
First we begin with a proposition that will be required on multiple occasions. This is an application of the local (spectral) large sieve (Lemma \ref{jutila}). 
\begin{proposition} \label{dis2}
For $1 \ll \vartheta \leq \frac{K}{2}$, $1 \ll K \ll X^2$, $\lambda \geq 0$ and $nl \leq m$, we have
\ba \nonumber
\sum_{\substack{ K + \vartheta \leq \kappa_j \leq K + 2\vartheta  }}| \rho_j(nr_1)\rho_j(m)|\frac{e^{-\pi |\kappa_j|}}{\kappa _j ^{\lambda}}
\ll r_1^{\theta + \ve} \frac{1}{K^{\lambda}}\l(K \vartheta + m\r) \l(mK\r)^{\ve}.
\ea
\end{proposition}  
\begin{proof}
Using \eqref{heckerelation} and \eqref{ramanujan}., 
\ba \nonumber
\sum_{\substack{ K + \vartheta \leq \kappa_j \leq K + 2\vartheta }}\frac{|\rho_j(nr_1)|^2}{\cosh(\pi \kappa_j)} 
\ll   r_1^{\theta + \ve} \sum_{k|(n,r_1)}\sum_{\substack{ K + \vartheta \leq \kappa_j \leq K + 2\vartheta }} \frac{|\rho_j(1)|^2}{\cosh(\pi \kappa_j)}   \l|\lambda _j\l(\frac{n}{k}\r)\r|^2 \quad,
\ea
We apply the Cauchy-Schwarz inequality and then Lemma \ref{jutila} and the condition $nl\leq m$ as well. 
\end{proof}
\begin{proposition} \label{D-0}
For $C_1 < C \leq \frac{8NX}{M}$, we have
\ba \nonumber
\mathfrak{D}_{-,0} \ll r_1^{\theta}X^{\ve}\l(\frac{C}{N^2l}\r)^{\frac{3}{4}} MN \sqrt{\frac{X}{H}}.
\ea 
\end{proposition}
\begin{proof}
Using \eqref{wei}, \eqref{lessthan1}, \eqref{calD-}, Prop. \ref{discrete} and recalling the discussion in Prop. \ref{ra1}, we have the proof.
\end{proof}
\noi
We are left to consider $\mathfrak{D}_{-,1}$ which is caried out in the next section. Note that Prop. \ref{c>c1} follows from the above proposition and Prop. \ref{D-1} stated below.
\begin{proposition}\label{D-1}
For $C_1 < C \leq \frac{8NX}{M}$, we have,
\ba \nonumber
\mathfrak{D}_{-,1} &\ll r_1^{\theta} X^{\frac{3}{4}+\ve} \sqrt{MN} \bigg (\l(\frac{C}{Xl}\r)^{\frac{1}{4}}\frac{1}{\sqrt{N}}+ \l(\frac{C}{Xl}\r)^{\frac{3}{4}}\frac{(M+N)}{N^{\frac{3}{2}}}+ \l(\frac{C}{Xl}\r)^{\frac{5}{4}}\frac{(MN)}{N^{\frac{5}{2}}} \bigg ) \\
&+ r_1^{\theta} X^{3/4+\ve} M ^{23/32} + r_1^{\theta}X^{\ve} \l(\frac{C}{N^2l}\r)^{\frac{3}{4}} MN \l( \sqrt{M} + \sqrt{N} + \frac{\sqrt{N}}{(MN)^{\frac{1}{4}}}\r).
\ea
\end{proposition}

\section{Analysis of  $\mathfrak{D}_{-,1}$}
We prove Prop \ref{D-1} in this section. 
In the range of the variables concerned, we have  $w\l(\frac{4\pi m\sqrt{mn}}{nu^2}\r)  =1$ and the integral in question becomes
\be \label{I-1} 
I_{-,1}:=2 r_1^{\nu + i\eta} \int\limits_{\alpha _1}^{\beta_ 1 } A(u) e(h_-(u)) \diff u, \text{ where} \alpha _1 := \frac{\sqrt{4\pi}\sqrt{m}(mn)^{1/4}}{\sqrt{Xn}},\text{ }
\beta = \beta_1 := \frac{\sqrt{4\pi} (mn)^{1/4}}{\sqrt{C_1}}
\ee 
Recall that the stationary point: $u_0 = \frac{2\eta}{\sqrt{\pi} X^{\frac{1}{2}}  (mn)^{\frac{1}{4}} }$. Here the analysis is somewhat similar to that of $I_4$ in \S \ref{conI4} but the main differences are that because of the presence of the weight function $A(u)$, we need to use the weighted stationary phase lemma (Lemma \ref{Huxley}). Also, apart from the large sieve inequality for twisted coefficients (Lemma \ref{ls}), here we also require the local (spectral) large sieve inequality of Jutila (Lemma \ref{jutila}). In order to analyze the integral $I_{-,1}$ (see \eqref{I-1}) and estimate $\mathfrak{D}_{-,1}$ we 
consider five intervals in which the stationary point $u_0$ can possibly lie, namely, $u_0 < \alpha _1 -\delta, u_0 > \beta _1 + \delta,  |u_0 - \alpha _1| \leq\delta,  |u_0 - \beta _1| \leq \delta,  \alpha _1 +\delta < u_0 < \beta _1 - \delta $, where we choose 
\be \label{deltavalue}
\delta = \frac{1}{100}\sqrt{\frac{M}{NX}},
\ee
and note that $\beta_1> 100\delta$ 
by our choice. We subdivide the sum $\mathfrak{D}_{-,1}$  accordingly as
\ba  \label{D-decomp}
\mathfrak{D}_{-,1}&= \mathfrak{D}_{-,1,1} + \cdots+ \mathfrak{D}_{-,1,5}.
\ea
 Furthermore, recall that the condition $\kappa_j \ll X^{1+\ve}/H$ holds here (see Prop \ref{ra1}). We introduce the notation
\be\label{notation}
c_{m,n}:=\frac{\sqrt{\pi}X^{\frac 1 2} (mn)^{\frac 1 4}}{2}.
\ee
\begin{proposition} \label{disc1}
For $C_1 < C \leq \frac{8NX}{M}$, we have
\ba \nonumber
&(i) \quad \mathfrak{D}_{-,1,1},\mathfrak{D}_{-,1,2}\ll r_1^{\theta}X^{\ve} \l(\frac{C}{N^2l}\r)^{\frac{3}{4}} MN \l(\sqrt{\frac{X}{H}} + \sqrt{M} + \frac{\sqrt{N}}{(MN)^{\frac{1}{4}}}\r). \\
&(ii) \quad \mathfrak{D}_{-,1,3} ,\mathfrak{D}_{-,1,4} \ll r_1^{\theta} X^{\ve} \l(\frac{C}{N^2l}\r)^{\frac{3}{4}} \l((MN)^{\frac{3}{4}} M+ MN\r).
\ea
\end{proposition}
\begin{proof}
Let us consider (i) first and we only estimate $\mathfrak{D}_{-,1,2}$  below as the other one can be estimated in the same way. The proof is similar to the proof of 
Prop. \ref{S_1U} and we only give a sketch. 
First note that since $\eta \ll X^{1+\ve}/H$, 
$
u_0 = \frac{2\eta}{\sqrt{\pi}X^{\frac{1}{2}}(mn)^{\frac{1}{4}}} \ll \frac{X^{\frac{1}{2}+\ve}}{H}.$
We write the range $\l[\beta_1+\delta, \frac{X^{\frac{1}{2}+\ve}}{H}\r]$ of $u_0$ as a union of dyadic subintervals of the form $[\beta _1+Y, \beta _1+2Y]$, where the dyadic variable $Y$ can take at most $O(\log X)$ values.
Let us consider two cases separately:  $\delta \leq Y \leq \frac{\beta_1}{2}$ and $Y > \frac{\beta_1}{2}$.\\
\textbf{Case I}: $\delta \leq Y \leq \frac{\beta_1}{2}$.
Suppose $\beta_1 + Y \leq u_0 \leq \beta_1 + 2Y$, where $\delta \leq Y \leq \frac{\beta_1}{2}$. By \eqref{h''u} and \eqref{k-der}, we have
\ba \label{fcheck-1} 
 \check{f}_{-,1}(\eta)  
&\ll \l(\frac{C}{N^2l}\r)^{\frac{3}{4}}  \frac{ 1}{Y X^{\frac{1}{2}} (MN)^{\frac{1}{4}}} \frac{1}{\sqrt{\eta}}e^{-\pi |\eta|} 
\ea 
Recall that $C \asymp C_1$ here. By \eqref{fcheck-1}, Prop. \ref{dis2}, \eqref{deltavalue}, We have 
\ba  \label{Di}
&\mathfrak{D}_{-,1,2}(I) := \mathop{\sum_ m \sum_ n}_{nl \leq m} b_1 \l(\frac{m}{M}\r)  b_2 \l(\frac{n}{N}\r)\sum_{\substack{ Y \textnormal{ dyadic } \\ \delta \leq Y \leq \frac{\beta_1}{2}}}  \sum_{  c_{m,n} (\beta_1 + Y) \leq \kappa _ j \leq c_{m,n}  (\beta _1 +2 Y) }  \rho_j(nr_1)\rho_j(m)\check{f}_{-,1}(\kappa_j) \\
 &\ll r_1^{\theta}X^{\ve} \l(\frac{C}{N^2l}\r)^{\frac{3}{4}} \mathop{\sum_ m \sum_ n}_{nl \leq m} b_1 \l(\frac{m}{M}\r)  b_2 \l(\frac{n}{N}\r) \l( \sqrt{M} + \frac{\sqrt{N}}{(MN)^{\frac{1}{4}}}\r).
\ea
\textbf{Case II}: $Y > \frac{\beta_1}{2}$.
In this case, the range of $\eta $ we need to consider is   (by \eqref{I-1} and  \eqref{h''u})
\be \label{3-4}
 \frac{3}{4}\sqrt{\pi} X^{\frac{1}{2}} (mn)^{\frac{1}{4}} \beta _1 \leq \eta \ll \frac{X^{1+\ve}}{H}.
\ee
Therefore, the contribution of $\mathfrak{D}_{-,1,2}$ for $\kappa _j$ in this range is given by
\begin{align}
&\mathfrak{D}_{-,1,2}(II) := \mathop{\sum_ m \sum_ n}_{nl \leq m} b_1 \l(\frac{m}{M}\r)  b_2 \l(\frac{n}{N}\r)\sum_{\substack{ Y \textnormal{ dyadic } \\  Y > \frac{\beta_1}{2}}} \sum_{\substack{\kappa_j \ll \frac{X^{1+\ve}}{H} \\ \beta_1 + Y \leq u_0 \leq \beta_1 + 2Y  }} \rho_j(nr_1)\rho_j(m)\check{f}_{-,1}(\kappa_j) \nonumber
\end{align}
We first bring the $Y$-sum and the $\kappa_j$-sum outside of the sums over $m$ and $n$. Since $\beta _1$ depends on $m,n$, we extend the range of $Y$ and $\kappa_j$ by positivity, as follows. We have, 
\ba \label{beta01000}
Y > \frac{\beta _1}{2} \implies Y \geq \beta _ 0 , \quad
\text{where}\quad
\beta_ 0 :=\frac{1}{1000} \frac{\sqrt{M}(MN)^{\frac{1}{4}}}{\sqrt{XN}}.
\ea
Thus, By \eqref{k-der}, \eqref{wei}, \eqref{I-1} and \eqref{lessthan1}, \eqref{beta01000}, \eqref{3-4} and  Prop. \ref{discrete}, we have
\ba\label{Dii}
\mathfrak{D}_{-,1,2}(II) 
&\ll  \l(\frac{C}{N^2l}\r)^{\frac{3}{4}}\l( r_1^{\theta} MN \l(\frac{X^{1+\ve}}{H}\r)^{\frac{1}{2}}\r).
\ea
The proof is finished in view of \eqref{Di} and \eqref{Dii}.\\
Now we consider part (ii).
Here,  $\eta$  is highly localized and Lemma \ref{Stein} with $k = 2$ suffices and we apply Prop. \ref{dis2} to handle the sum over $\kappa_j$. We omit the details.
\end{proof}

\subsection{Estimation of $\mathfrak{D}_{-,1,5}$ and the weighted stationary phase analysis } \label{frakDsection}
\begin{proposition}\label{disc4}
For $C_1 < C \leq \frac{8NX}{M}$, we have, 
\ba \nonumber
&\mathfrak{D}_{-,1,5} \ll r_1^{\theta} X^{\frac{3}{4}+\ve} \sqrt{MN} \bigg (\l(\frac{C}{Xl}\r)^{\frac{1}{4}}\frac{1}{\sqrt{N}}+ \l(\frac{C}{Xl}\r)^{\frac{3}{4}}\frac{(M+N)}{N^{\frac{3}{2}}}+ \l(\frac{C}{Xl}\r)^{\frac{5}{4}}\frac{(MN)}{N^{\frac{5}{2}}} \bigg )
+ {r_1}^{\theta} X^{\frac{3}{4}+\ve}N \\
&+ r_1^{\theta} X^{\frac{3}{4}+\ve} M ^{\frac{23}{32}}  + r_1^{\theta}X^{\ve} \l(\frac{C}{N^2l}\r)^{\frac{3}{4}} MN \l( \sqrt{M} + \sqrt{N} + \frac{\sqrt{N}}{(MN)^{\frac{1}{4}}}\r)
 + r_1^{\theta} X^{\ve}\l(\frac{C}{N^2l}\r)^{\frac{3}{4}} \l((MN)^{\frac{3}{4}} M + MN\r) .
\ea
\end{proposition}
We extract the stationary phase of the integral $I_{-,1}$ (see \eqref{I-1}) by Lemma \ref{Huxley} as follows:
\ba \nonumber
I_{-,1}&= 2 r_1^{\nu+i\eta} \int\limits_{\alpha _1}^{\beta_ 1 } A(u) e(h_-(u)) \diff u 
 = 2 r_1^{\nu + i\eta} \frac{A(u_0)}{\sqrt{h''_-(u_0)}} e\l(-h_-(u_0)-\frac{1}{8}\r)+O(R_1+R_2+R_3),
\ea
where $\alpha_1$ and $\beta_1$ are defined in \eqref{I-1} and 
\ba \label{R1huxley}
R_1 := \frac{|A(\beta_1)|}{|h_-'(\beta _1)|} + \frac{|A(\alpha_1)|}{|h_-'(\alpha _1)|},\text{ }
R_2 := \frac{Q^4T}{P^2} \l(1+ \frac{Q}{V} \r)^2\l(\frac{1}{(x_0 - \alpha_1)^3}+ \frac{1}{( \beta_1 - x_0 )^3} \r),\text{ and }
R_3 := \frac{QT}{P^{\frac{3}{2}}}\l(1+ \frac{Q}{V} \r)^2,
\ea
where $Q,T,P$ and $V$ are defined in Lemma \ref{Huxley}. It is clear that the leading term here is $\mathcal{L}$ (see \eqref{lt}). Therefore,
\ba \nonumber
\mathfrak{D}_{-,1,5}& = 4 \mathop{\sum_ m \sum_ n}_{nl \leq m}b_1 \l(\frac{m}{M}\r)  b_2 \l(\frac{n}{N}\r)
\sum_{\nu = 0}^{\infty}  
\sum_{\substack{ \kappa_j \\  \alpha _1 +\delta < u_0 < \beta _1 - \delta } }
 \rho_j(nr_1)\rho_j(m)  \mathcal{V} \Big[\mathcal{L}+ O(R_1+R_2+R_3)\Big].
\ea
We rewrite the above sum as
\ba  \label{illish}
& \mathop{\sum_ m \sum_ n}_{nl \leq m}\dots 
\sum_{\nu = 0}^{\infty}  
\l\{\sum_{\substack{ \kappa_j \\  \alpha _1  < u_0 < \beta _1 } }-\Bigl(\sum_{\substack{ \kappa_j \\  \alpha _1  < u_0 < \alpha _1 +\delta } }+ \sum_{\substack{ \kappa_j \\  \beta_1 -\delta < u_0 < \beta_1 } }\Bigr) +O\Bigl(\sum_{\substack{ \kappa_j \\  \alpha _1 +\delta < u_0 < \beta _1 - \delta } }(R_1+R_2+R_3)\Bigr)
\r\}\\
&=G-H+J.
\ea

\subsubsection{\textbf{Contribution of $G$}}
Here, the condition $ \alpha _1 < u_0 < \beta _1$ translates to:
\be \label{extra}
\pi m <\kappa_j< \pi \sqrt{\frac{mnX}{C_1}},
\ee
by \eqref{h''u} and \eqref{I-1}.
This sum is similar to the sum in \eqref{D_0-expression} except for the  condition \eqref{extra}. In \eqref{D_0-expression}, the corresponding condition on $\kappa_j$ was automatic from the support of the function $V$. Here also the presence of the function $V$ in $ \mathcal{L}$ limits  $\kappa_j$ to lie in the interval given by \eqref{eta2}. But we have an extra condition $\kappa_j>\pi m$  from which we need to separate the variables $m$ and $\kappa_j$. By Lemma \ref{perron},
\ba\label{removal-2}
G&=\sum_{\kappa_j >\pi m}(\cdots)=\sum_{\kappa_j} (\cdots)\frac{1}{2\pi i} \int\limits _{\sigma- i T}^{\sigma + iT} \l(\frac{\kappa_j}{\pi m}\r)^s \frac{\diff s}{s}-\frac 1 2 \sum_{\kappa_j} (\cdots)\mathbf{1}_{\kappa_j=\pi m} \\
&+O\l(\sum_{\kappa_j} |(\cdots)|\l( \l(\frac{\kappa_j}{\pi m}\r)^{\sigma}\min\l(1, \frac 1 {T|\log(\frac{\kappa_j}{\pi m})|} \r)+\frac 1 T\r) \r).
\ea
Here,
\[
(\cdots)=4 \mathop{\sum_ m \sum_ n}_{nl \leq m}b_1 \l(\frac{m}{M}\r)  b_2 \l(\frac{n}{N}\r)
\sum_{\nu = 0}^{\infty}  
 \rho_j(nr_1)\rho_j(m)\mathcal{V}(\kappa _j, \nu , r_1)\mathcal{L}(m,n,\kappa _ j, \nu, C),
 \]
 and recall that the support of $V$ inside $\mathcal{L}(m,n,\kappa _ j, \nu, C)$ forces $\kappa_j$ to lie in the  interval given by \eqref{eta2} which is independent of $m$ and $n$. 
Now, for  estimating the contribution of the first term, we insert the sum over $\kappa_j$ inside the integral and estimate the 
sum as in \S \ref{leading-term}. The only difference is the presence of the factors ${\kappa_j}^s$ and $m^{-
s}$ which does not make any difference since we will take $\sigma$ to be a  small positive real number $\ve$. Thus the contribution of the integral over $s$ in the sum $G$ will be the same as the contribution of the leading term $\mathcal{D}_{-, 0}$ in $\mathcal{D}_{-}$ (see  \eqref{dis1}) for $C_0 \leq C \leq C_1$, up to a factor of size $O(X^{\ve})$. Also, we choose $T$ to be a large power of $X$ (say, $X^{10}$)  as in \S \ref{removal} and this only adds a factor of size $\log X$. Thus, we have the bound (from \eqref{dis1})
\ba \label{G-leading}
& 4\mathop{\sum_ m \sum_ n}_{nl \leq m}b_1 \l(\frac{m}{M}\r)  b_2 \l(\frac{n}{N}\r)
\sum_{\nu = 0}^{\infty}  
\sum
 \rho_j(nr_1)\rho_j(m)\mathcal{V} \mathcal{L}
 \frac{1}{2\pi i}\int\limits _{\sigma- i T}^{\sigma + iT} 
\l(\frac{\kappa_j}{\pi m}\r)^s
 \frac{\diff s}{s}\\
&\ll r_1^{\theta} X^{\frac{3}{4}+\ve} \sqrt{MN} \bigg (\l(\frac{C}{Xl}\r)^{\frac{1}{4}}\frac{1}{\sqrt{N}}+ \l(\frac{C}{Xl}\r)^{\frac{3}{4}}\frac{(M+N)}{N^{\frac{3}{2}}}+ \l(\frac{C}{Xl}\r)^{\frac{5}{4}}\frac{(MN)}{N^{\frac{5}{2}}} \bigg )+ r^{\theta} X^{\frac{3}{4}+\ve}N.
\ea
By \eqref{lt}, \eqref{nu-bound-original} and \eqref{lessthan1}, the contribution of the term $\frac 1 2 \mathbf{1}_{\kappa_j=\pi m} $ 
is bounded by
\ba \label{kappa=m}
& X^{3/4} \mathop{\sum_ {m \sim M}\sum_ {n \sim N}}_{nl \leq m} \sum_{\substack{ \kappa_j \\ \kappa_j = \pi m} } \frac{|\rho_j(nr_1)\rho_j(m)| e^{-\pi \kappa_j }}{\kappa_j ^{5/2}}
\ll r_1^{\theta} X^{3/4+\ve} M ^{23/32} \quad  \text{(by Lemma \ref{weyl}), \eqref{ramanujan} and Remark \ref{kim})}.
\ea
Now we estimate the contribution of the error terms in \eqref{removal-2}. 
The term $O(1/T)$ presents no problem, and as for the other error
term, we subdivide the sum over $\kappa_j$ into two parts: $|\kappa_j -\pi m|>1$ and $|\kappa_j -\pi m|\leq1$.
In the first case, we observe that
\[
1/|\log (\kappa_j/\pi m)| \ll m,
\]
with an absolute implied constant. Since $T$ is very large, the contribution of this part is negligibly small.
For the other case, we take the 
 minimum in the $O$-term to be $1$ and  use the Weyl law to bound the resulting sum and the calculation is identical to that of the case $\kappa_j= \pi m$ treated above. Thus the contribution of the $O$-terms is also $O\l(r_1^{\theta} X^{3/4+\ve} M ^{23/32}\r)$. 
\subsubsection{\textbf{Contribution of $H$ }}
Using \eqref{lt} and \eqref{wei}, \eqref{nu-bound-original} and the same analysis as in the proof of part (ii), we conclude that the contribution of these two terms is
\be \label{othererror}
O\l( r_1^{\theta} X^{\ve}\l(\frac{C}{N^2l}\r)^{\frac{3}{4}} \l((MN)^{\frac{3}{4}} M + MN\r)\r). 
\ee
\subsubsection{\textbf{Contribution of $J$}}
We consider $R_1$ first and  only estimate the first term since the second term can be handled in a similar way. By making a dyadic subdivision of the range $\l( \alpha _1+ \delta , \beta_1-\delta \r)$ of $u_0$ into $O(\log X)$ many subintervals and proceeding as before, and using \eqref{wei} and \eqref{au}, we see that 
the contribution of  the term $\frac{|A(\beta_1)|}{|h_-'(\beta _1)|}$ in 
$\mathfrak{D}_{-,1,5}$ is bounded by 
\ba \label{error1}
& \l(\frac{C}{N^2l}\r)^{\frac{3}{4}} \mathop{\sum_ m \sum_ n}_{nl \leq m}\cdots \sum_{\substack{ Y\text{ dyadic} \\ \delta \leq Y \leq \frac{\beta _1}{2} }}  \sum_{c_{m,n}(\beta_1 - 2Y) < \kappa_j <c_{m,n}(\beta _1 - Y) } 
 |\rho_j(nr_1)\rho_j(m)| \frac{ 1}{Y X^{\frac{1}{2}} (MN)^{\frac{1}{4}}} \frac{1}{\sqrt{\kappa _j}}e^{-\pi |\kappa _j|} \\
& \ll r_1^{\theta}X^{\ve} \l(\frac{C}{N^2l}\r)^{\frac{3}{4}} MN \l( \sqrt{M} + \frac{\sqrt{N}}{(MN)^{\frac{1}{4}}}\r) \quad (\text{by Prop. \ref{dis2}, \eqref{deltavalue} and \eqref{I-1}}).
\ea
For $R_2$ and $R_3$, we first note that 
\ba \nonumber
Q = V = \beta _ 1 , P = \kappa_j \text{ and }T = \frac{\l(\frac{\sqrt{MN}}{C}\r)^{2\nu}}{\beta _ 0} \l(\frac{C}{N^2l}\r)^{\frac{3}{4}},
\ea
where $\beta_0$ and $\beta _1$ is defined in \eqref{beta01000} and \eqref{I-1} respectively. 
The contribution of $R_2$ is handled in the same way as above and we make use of the facts that $u_0 \asymp \beta_1 \asymp \alpha_1$ and
$nl\leq m$. Thus, the contribution of  $R_2$ and $R_3$ in $\mathfrak{D}_{-,1,5}$ is 
\be \label{error2}
O\l( r_1^{\theta}  X^{\ve}\l(\frac{C}{N^2l}\r)^{\frac{3}{4}} MN\l( \sqrt{N}+1 \r) + r_1^{\theta} X^{\ve}\l(\frac{C}{N^2l}\r)^{\frac{3}{4}} MN\r).
 \ee
Thus we complete the proof of Prop. \ref{disc4} by clubbing together \eqref{illish}, \eqref{G-leading}, \eqref{kappa=m}, \eqref{error1}, \eqref{error2}, \eqref{othererror} and recalling \eqref{dya}.
Combining \eqref{D-decomp}, Prop. \ref{disc1}, and Prop. \ref{disc4}, Prop. \ref{D-1} follows.

\section{Esitimation of $\mathcal{C}_{-}$ for $C \geq C_0$}\label{section-large-C} 
\begin{proposition} \label{C>C0}
For $C \geq C_0$, we have 
\ba \nonumber
\mathcal{C}_- \ll \l(\frac{C}{N^2l}\r)^{\frac{3}{4}} MNX^{\ve}.
\ea
\end{proposition}
\begin{proof}
By \eqref{k-der} with $k=2$ for the integral $I_-$, we have $\check{f}_{-}(\eta) 
 \ll \l(\frac{C}{N^2l}\r)^{\frac{3}{4}} \frac{e^{-\pi|\eta|}}{\eta}$. By Prop \ref{ra1}, Prop. \ref{prop4} and by \eqref{Cpm}, we have the bound.
\end{proof} 
\section{Proof of Proposition \ref{second prop}: final steps}\label{last}
\subsection{Proof of Proposition \ref{kuz1}}
Bringing together the estimates for $\mathcal{D}_+$, $\mathcal{C}_+$, $\mathcal{D}_-$ and $\mathcal{C}_-$, we prove Prop. \ref{kuz1}.
\subsection{A sketch of the proof of Proposition \ref{kuz2}}
Note that here the role of the function $f(t)$  (see \eqref{f-function})
is now played by the function 
\[
f_{m,n,C}^{\star} (t)=w\l(\frac{4\pi m \sqrt{mnr_1}}{nt}\frac u X \r)V\l(\frac{4\pi \sqrt{mnr_1}}{tC}\r) g^{\star}\l(m,n,\frac{4\pi \sqrt{mnr_1}}{t}\r)e\l(- \frac{ u^{\frac{1}{2}} t^{\frac{1}{2}}(mn)^{\frac{1}{4}} }{\sqrt{\pi}r_1^{1/4}}\r). 
\]
Note that  the phase function now is essentially the same as before with the only change being that the parameter 
$X$ has now been replaced by $u$, which is of the same size as $X$. 
Because the function $g$ 
is now replaced by $g^{\star}$ (see \eqref{g1}), instead of the function $A(u)$ we now have a corresponding amplitude function
$A^{\star}(u)$ which satifies the bound
\[
\|A^{\star} \|_{\infty} \ll \frac{\l(\frac{\sqrt{MN}}{C}\r)^{2\nu}}{\l(\frac{\sqrt{MN}}{C}\r)^{\frac{1}{2}}}\l(\frac{C}{N^2l}\r)^{\frac{1}{4}}. 
\]
The proofs of the analogues of Prop. \ref{D+}, Prop. \ref{main-d-minus}, Prop. \ref{c>c1}, and Prop. \ref{C>C0} are along the same lines but with some minor changes. For example, the bound for  $\mathcal{D}_{-}$ in Prop. \ref{main-d-minus} (for $C\leq C_1$) is now replaced by 
$\l(r^{\theta} \l(X^{\frac{1}{4}+\ve}\l(X/H\r)^2 +X^{\frac 1 3}\l(X/H \r)^{\frac 3 2}\r)\r)$. To see this, note that much of the proof of the proposition carries over but now $\mathcal{L}(m,n,\kappa _ j, \nu, C)$ (see \eqref{lt}) is
changed to $\mathcal{L}^{\star}(m,n,\kappa _ j, \nu, C) $ where the factor  $\l(\frac{{\pi}^2 X}{{\eta}^2}\r)^{\frac 3 4}$ is now replaced by $\l(\frac{{\pi}^2 X}{{\eta}^2}\r)^{\frac 1 4}$. Note that the main contribution to the bound in  Prop. \ref{main-d-minus} comes from the term $E_2$ (see Prop. \ref{removalprop}), $\mathcal{D}_{-,0}$ (see \eqref{dis1}), and $\mathcal{D}_{-,k}$ (see \S \ref{estimate4}). The term analogous to $E_2$ now contributes $O\l(r^{\theta}X^{\frac 1 3}\l(X/H\r)^{\frac 3 2}\r)$ and the same contribution also comes from the analogue of $\mathcal{D}_{-,0}$. In both the cases, the exponent of $X$ changes from $\frac 3 4$ to $\frac 1 4$ and the exponent of $\kappa_j$ changes from $\frac 5 2$ to $ \frac 3 2$. 
The analogue of $\mathcal{D}_{-,k}$ now contributes
$\l(r^{\theta} X^{\frac{1}{4}+\ve}\l(X/H\r)^2 \r)$ since the factor $\l(\frac{C}{N^2l}\r)^{\frac{3}{4}}$
is now changed to $\l(\frac{C}{N^2l}\r)^{\frac{1}{4}}$. We obtain the bound $O\l(r^{\theta}X^{\ve} \l(X^{{\frac 1 4 }}\l(X/H\r)^{2} +X^{\frac 1 4}\l(X/H \r)^{\frac 3 2}\r)\r)$ for the analogue of $\mathcal{D}_{-}$ in Prop. \ref{c>c1}
(for $C>C_1$) as well. For the analogue of $\mathcal{D}_{+}($ in Prop. \ref{D+}, we now have the bound $O\l(r^{\theta} X^{\frac{1}{4}+\ve}\l(X/H\r)^2\r)$. Together, the analogous of $\mathcal{C}_{-}$ and $\mathcal{C}_{+}$ in Prop.s \ref{C>C0} and \ref{D+} now contribute 
$O\l(X^{{\frac 1 4}+\ve} {\l(X/H\r)}^{\frac 3 2}\r)$.
\subsection{Proof of Proposition \ref{second prop}}
The proof of Prop. \ref{second prop} follows from the above discussion and Prop. \ref{ET}, Prop. \ref{kuz1}, Prop. \ref{kuz2}.
\begin{remark}\label{1/3}
We explain here why we need to choose $r=O\l(X^{1/3}\r)$. Note that from \eqref{lessthan1}, we have the condition
$r_1\ll {C_0}^2/MN$. For $R_3(M,N,C, X, u)$, we similarly have the condition $r_1\ll {C_0^{\star}}^2/MN$.
Since we need to choose $C_0^{\star}\leq \frac{X^{2/3}+\ve}{l}$ (see \eqref{c01}), we get $r\ll X^{1/3}$. 
\end{remark}

\section{Proofs of Theorem \ref{arbitrary}, Theorem \ref{weighted} and Corollary \ref{special}}

 Theorem \ref{arbitrary} 
is a direct consequence of  \eqref{SrX}, \eqref{SW-decomp} and Remark \ref{5/3}.\\
Now we prove Theorem \ref{weighted}.
\begin{proof}
Expressing the periodic functions $f_1$ and $f_2$ in terms of the additive characters modulo $q_1$ and $q_2$ respectively, 
 we obtain
 \[
\abcdsumoneX \alpha(a)f_1(b)f_2(c) =\sum_{\beta (\text{mod }q_1)}\sum_{\gamma (\text{mod }q_2)}\widehat{f_{1, q_1}}(\beta)\widehat{f_{2, q_2}}(\gamma)S_{\beta, \gamma} (X), 
 \]
 where
\[
S_{\beta, \gamma} (X):=\abcdsumoneX\alpha(a) e(b\beta/q_1 +c\gamma/q_2).
\]
We isolate the sum for  $(\beta, \gamma)=(0, 0)\in \Z/q_1\Z \times \Z/q_2\Z$ and this sum is the same as the one in  Theorem \ref{arbitrary} with $r=1$ since both the additive characters become trivial.  This yields the first main term. Among the remaining pairs we consider only the case where   both $\beta$ and $\gamma$ are non-zero. If one of them is the zero residue class then there is only one non-trivial additive character and this case is easier to handle and  such sums  contribute to the error term.   For  $(\beta, \gamma)\in \l(\Z/q_1\Z\r)^{\times} \times\l( \Z/q_2\Z\r)^{\times}$, we proceed to apply the same steps as in the proof of  Theorem \ref{main} to the sum $S_{\alpha, \beta}$.  By Poisson summation on the sums over $b$ and $c$,  we obtain,
\ba \nonumber
S_{\beta, \gamma} (X)&=\sum_{a} \frac{\alpha (a)w(a)}{a^2} \mathop{\sum\sum}_{m, n \in \Z} S(n, -m; a)\widehat{G}\l(\frac m a -\frac{\beta}{q_1}, \frac n a -\frac{\gamma}{q_2} \r)\\
&=\sum_{a: (a, q_1 q_2)=1} \frac{\alpha (a)w(a)}{a^2} \mathop{\sum\sum}_{\substack{m\equiv -a\beta (\text{mod }q_1)\\ n\equiv -a\gamma (\text{mod }q_2)}} S(n\overline{q_2}, -m \overline{q_1}; a)\widehat{G}\l(\frac m {aq_1} ,\frac n {aq_2}\r)\\
&+ \sum_{a: (a, q_1)=1, q_2|a} \frac{\alpha (a)w(a)}{a^2} \mathop{\sum\sum}_{m\equiv -a\beta (\text{mod }q_1),  n }S(n, -m\overline{q_1}; a)\widehat{G}\l(\frac m {aq_1} ,\frac n a -\frac{\gamma}{q_2}\r)\\
&+ \sum_{a: (a, q_2)=1, \ q_1|a} \frac{\alpha (a)w(a)}{a^2} \mathop{\sum\sum}_{n\equiv -a\gamma (\text{mod }q_2),  m }S(n\overline{q_2}, -m; a)\widehat{G}\l(\frac m a -\frac{\beta}{q_1} ,\frac{n}{aq_2}\r)\\
&+ \sum_{a\equiv 0 (\text{mod }q_1 q_2)} \frac{\alpha (a)w(a)}{a^2} \mathop{\sum\sum}_{m, n} S(n+a\gamma/q_2, -(m+a\beta/q_1); a)\widehat{G}\l(\frac m {a} ,\frac n {a}\r).
\ea
Note that for the four sums on the right hand side, as long as one of $m$ and $n$ is non-zero, then our method yields the bound $O(X^{5/3+\ve})$  on choosing $H=X^{2/3}$, as in Remark \ref{5/3}. Now, owing to the congruence conditions appearing in the sums, both $m$ and $n$ can be simultaneously zero only in the fourth sum and the contribution of the term $(m,n)=(0,0)$ is
\ba
S_{\beta, \gamma} &(X, 0, 0)=\sum_{a\equiv 0 (\text{mod }q_1 q_2)} \frac{\alpha (a)w(a)}{a^2} S(a\gamma/q_2, -a\beta/q_1; a)\widehat{G}\l(0, 0\r) \nonumber \\
&=\sum_{j_1=1}^{\infty}\sum_{j_2=1}^{\infty}\sum_{(a,q_1 q_2)=1} \frac{\alpha (a{q_1}^{j_1}{q_2}^{j_2})w(a{q_1}^{j_1}{q_2}^{j_2})}{(a{q_1}^{j_1}{q_2}^{j_2})^2} S(a{q_1}^{j_1}{q_2}^{j_2}\gamma/q_2, -a{q_1}^{j_1}{q_2}^{j_2}\beta/q_1; a{q_1}^{j_1}{q_2}^{j_2})\widehat{G}\l(0, 0\r).\nonumber
\ea
Using the twisted multiplicativity of  Kloosterman sums and simplifying the resulting expressions and noting that 
\[
\widehat{G}\l(0, 0\r)=\int\int w(x)w(y)w\l(xy/ a{q_1}^{j_1}{q_2}^{j_2}\r) \diff x \diff y,
\]
we get
\ba\label{zerozero}
S_{\beta, \gamma} (X, 0, 0)=\int\int w(x)w(y)\sum_{(a,q_1 q_2)=1} \frac{\phi(a)}{a^2}\sum_{j_1=1}^{\infty}\sum_{j_2=1}^{\infty}\frac{\alpha (a{q_1}^{j_1}{q_2}^{j_2})w(a{q_1}^{j_1}{q_2}^{j_2})}{{q_1}^{j_1+1}{q_2}^{j_2+1}} w\l(xy/ a{q_1}^{j_1}{q_2}^{j_2}\r)\diff x \diff y.
\ea
This finishes the proof after we observe that the above expression is independent of $\beta$ and $\gamma$ and 
\[
\starsum_{\delta (\text{mod }q_j)} \widehat{f_{j, q_j}} (\delta)=f_j (0)-\widehat{f_{j, q_j}} (0
)
\]
for $j=1, 2$, where $\starsum$ denotes, as usual, the sum over the non-zero residue classes. 
\end{proof}
\subsection{Proof of Cor. \ref{special}}
\begin{proof}
We need a finer analysis to extract a precise main term from $S_{\beta, \gamma} (X, 0, 0)$ where $\alpha$ is now identically equal to $1$. First we replace the   weight function $w$  by $w_H(x):=w(x-H)$  so that the support of the new weight function contains only positive real numbers. This replacement introduces an acceptable error term $O(X^{1+\ve}H)=O(X^{5/3+\ve})$ as seen by trivial estimation.
In \eqref{zerozero}), putting $\alpha(a)=1$, using Mellin inversion to express the values of the function $w_H$, expressing the sum over $a$ in terms of the Riemann zeta function and executing the sums over $j_1$ and $j_2$, we have, 
\ba\nonumber
S_{\beta, \gamma} (X, 0, 0)= \int\int w_H(x)w_H(y)\int_{(c_1)}\int_{(c_2)}&\frac{ \widetilde{w_H}(s_1)\widetilde{w_H} (s_2)}{(xy)^{s_2}}\frac{\zeta(1+s_1-s_2)}{\zeta(2+s_1 -s_2)}\frac{1}{(q_1 q_2)^{3+s_1 -s_2}} \\
&\times {\l(1-\frac{1}{{q_1}^{2+s_1 -s_2}}\r)}^{-1}{\l(1-\frac{1}{{q_2}^{2+s_1 -s_2}}\r)}^{-1}\diff s_1 \diff s_2\diff x \diff y,
\ea
where $c_1> c_2 \geq 1$ and $\widetilde{w_H}$ denotes the Mellin transform of $w_H$. We shift the $s_1$-contour to the left  up to the line
\(
\Re(s_1) = c_2 - \delta,
\)
for a suitable $\delta$ with \( 0 < \delta < c_1 - c_2\),
crossing a simple pole at $s_1 = s_2$.  The resulting residue is 
\[
\frac{1}{\zeta(2)q_1 q_2 ({q_1}^2 -1) ({q_2}^2 -1) }\int\int w_H(x)w_H(y) \l(\int_{(c_2)}\frac{ (\widetilde{w_H}(s_2))^2}{(xy)^{s_2}} \diff s_2\r) \diff x \diff y.\\
\]
Now using Mellin inversion again, the last contour integral is 
\[
(w_H\ast w_H)(xy)= \int w_H(t) w_H(xy/t) \frac{1}{t}\diff t.
\]
Finally, we find that  the residue is
\[
\frac{1}{\zeta(2)q_1 q_2 ({q_1}^2 -1) ({q_2}^2 -1)} \int\int \int w_H(x)w_H(y)  w_H(t) w_H(xy/t)\frac{1}{t}\diff x \diff y \diff t.
\]
To study the above triple integral $\mathcal{I}$, we first observe that if any of the variables is restricted to lie in an interval of length $H$ then the resulting truncated integral is $O(XH)$ which is acceptable to us. Therefore, we may approximate the function $w_H$ by the indicator function of the interval $[1, X]$ and write
\ba\nonumber
\mathcal{I}&=\l(\int\limits_{1 \leq t \leq X} \frac {1}{t} \int\limits_{t \leq y \leq X} \int\limits_{1 \leq x \leq {\frac{tX}{y}}}1 + \int\limits_{1 \leq t \leq X} \frac {1}{t} \int\limits _{1 \leq y \leq t } \int\limits_{\frac{t}{y} \leq x \leq X} 1\r) \diff x\diff y \diff t +O\l(X^{5/3+\ve}\r)\\
&=2X^2+O\l(X^{5/3+\ve}\r)\nonumber.
\ea
This contributes to the main term. 
The shifted contour contributes
\begin{align}
	E =\int_{(c_2)} \int_{\Re(s_1)=c_2-\delta}
	 \widetilde{w_H}(s_1)\widetilde{w_H}(s_2)\widetilde{w_H}(1-s_2)^2
	 \frac{\zeta(1+s_1-s_2)}{\zeta(2+s_1-s_2)}
	 \frac{1}{(q_1 q_2)^{3+s_1-s_2}} 
	 \prod_{j=1}^2\left(1-\frac{1}{q_j^{2+s_1-s_2}}\right)^{-1} \diff s_1 \diff s_2.
\end{align}
By integration by parts, 
\ba 
\widetilde{w_H}(s) \ll \min \l(\frac{X^{\Re{s}}}{|s|}, \quad   \frac{X^{\Re{s}+1} }{H|s||s+1|}\r).
\ea
On using the first bound for $\widetilde{w_H}(s_2)$, the  second one for $\widetilde{w_H}(s_1)$, and on using standard bounds for  $\zeta(s)$, we  have
\begin{align}
	E \ll \frac{X^{3-\delta+\ve}}{H(q_1 q_2)^3 },
\end{align}
Choosing $\delta = 2/3$, we finally  have 
\ba 
E \ll X^{5/3+\ve},
\ea 
and this finishes the proof.
\end{proof}

\bibliographystyle{amsalpha}
\bibliography{reference}
\end{document}